\documentclass[a4paper,11pt]{article}%
\usepackage{amssymb}
\usepackage{graphicx}
\usepackage{amsmath}%
\usepackage{amsfonts}
\providecommand{\U}[1]{\protect\rule{.1in}{.1in}}
\newtheorem{theorem}{Theorem}[section]

\newtheorem{corollary}[theorem]{Corollary}

\newtheorem{lemma}[theorem]{Lemma}

\newtheorem{proposition}[theorem]{Proposition}
\newtheorem{remark}[theorem]{Remark}

\numberwithin{equation}{section}
\newenvironment{proof}[1][Proof]{\textbf{#1.} }{\ \rule{0.5em}{0.5em}}
\begin{document}

\title{Asymptotic Equivalence of Spectral Density \\[0.3cm] Estimation and Gaussian White Noise}
\author{\textsc{Georgi K. Golubev, Michael Nussbaum}$^{1}$\textsc{\bigskip\ and
Harrison H. Zhou}$^{1}$\\\textit{Universit\'{e} de Provence, Cornell University and Yale University }}
\date{}
\maketitle

\begin{abstract}
We consider the statistical experiment given by a sample $y(1),\ldots,y(n)$ of
a stationary Gaussian process with an unknown smooth spectral density $f$.
Asymptotic equivalence, in the sense of Le Cam's deficiency $\Delta$-distance,
to two Gaussian experiments with simpler structure is established. The first
one is given by independent zero mean Gaussians with variance approximately
$f(\omega_{i})$ where $\omega_{i}$ is a uniform grid of points in $(-\pi,\pi)$
(nonparametric Gaussian scale regression). This approximation is closely
related to well-known asymptotic independence results for the periodogram and
corresponding inference methods. The second asymptotic equivalence is to a
Gaussian white noise model where the drift function is the log-spectral
density. This represents the step from a Gaussian scale model to a location
model, and also has a counterpart in established inference methods, i.e.
log-periodogram regression. The problem of simple explicit equivalence maps
(Markov kernels), allowing to directly carry over inference, appears in this
context but is not solved here.

\end{abstract}

\begin{table}[b]
\rule{3cm}{0.01cm} \newline
\par
$^{1}${\small Supported in part by NSF\ Grant DMS-03-06497 \newline1991
}\newline{\small \textit{Key
words and phrases.} Stationary Gaussian process, spectral density, Sobolev
classes, Le Cam distance, asymptotic equivalence, Whittle likelihood,
log-periodogram regression, nonparametric Gaussian scale model, signal in
Gaussian white noise.}\end{table}

\section{Introduction and main results}

Estimation of the spectral density $f(\omega)$, $\omega\in\lbrack-\pi,\pi]$ of
a stationary process is an important and traditional problem of mathematical
statistics. We observe a sample $y^{(n)}=\left(  y(1),\ldots,y(n)\right)
^{\prime}$ from a real Gaussian stationary sequence $y(t)$ with $\mathrm{E}%
y(t)=0$ and autocovariance function $\gamma(h)=\mathrm{E}y(t)y(t+h)$. Consider
the spectral density, defined on $[-\pi,\pi]$ by
\begin{equation}
f(\omega)=\frac{1}{2\pi}\sum_{h=-\infty}^{\infty}\gamma(h)\mathrm{\exp
}{\normalsize (\mathrm{i}h\omega)}\label{first-equa}%
\end{equation}
where it is assumed that $\sum_{h=-\infty}^{\infty}\gamma^{2}(h)<\infty$. Let
$\Gamma_{n}$ be the $n\times n$ Toeplitz covariance matrix associated with
$\gamma(\cdot),$ i.e. the matrix with entries
\begin{equation}
\left(  \Gamma_{n}\right)  _{j,k}=\gamma(k-j)=\int_{-\pi}^{\pi}\exp\left(
\mathrm{i}\,(k-j)\omega\right)  f(\omega)\mathrm{\,d}\omega,\qquad
j,k=1,\ldots,n.\label{gamma-n-def}%
\end{equation}
Write $\Gamma_{n}(f)$ for the covariance matrix corresponding to spectral
density $f$ and note that $y^{(n)}$ has a multivariate normal distribution
$N_{n}(0,\Gamma_{n}(f))$. Let $\Sigma$ be a nonparametric set of spectral
densities to be described below. We are interested in the approximation of the
statistical experiment
\begin{equation}
\mathcal{E}_{n}=\left(  N_{n}(0,\Gamma_{n}(f)),f\in\Sigma\right)
\label{exper-1}%
\end{equation}
in the sense of Le Cam's deficiency pseudodistance $\Delta(\cdot,\cdot)$; see
the end of this section for a precise definition. The statistical
interpretation of the Le Cam distance is as follows. For two experiments
$\mathcal{E}$ and $\mathcal{F}$ having the same parameter space,
$\Delta(\mathcal{E},\mathcal{F})<\varepsilon$ implies that for any decision
problem with loss bounded by $1$ and any statistical procedure with the
experiment $\mathcal{E}$ there is a (randomized) procedure with $\mathcal{F}$
the risk of which evaluated in $\mathcal{F}$ nearly matches (within
$\varepsilon$) the risk of the original procedure evaluated in $\mathcal{E}$.
In this statement the roles of $\mathcal{E}$ and $\mathcal{F}$ can also be
reversed. Two sequences $\mathcal{E}_{n},\mathcal{F}_{n}$ are said to be
\textit{asymptotically equivalent} if $\Delta(\mathcal{E}_{n},\mathcal{F}%
_{n})\rightarrow0$.

As a guide to what can be expected, consider first the case where
$f_{\vartheta}$, $\vartheta\in\Theta$ is a smooth parametric family of
spectral densities. Assume that $\Theta$ is a real interval; under some
regularity conditions, the model is well known to fulfill the standard
LAN\ conditions with localization rate $n^{-1/2}$ and normalized Fisher
information at $\vartheta$%
\[
\frac{1}{4\pi}\int_{-\pi}^{\pi}\left(  \frac{\partial}{\partial\vartheta}\log
f_{\vartheta}(\omega)\right)  ^{2}d\omega
\]
(Davies (1973), Dzhaparidze (1985), chap. I.3, cf. also the discussion in van
der Vaart (1998), Example 7.17). Consider the parametric Gaussian white noise
model where the signal is the log-spectral density:
\begin{equation}
dZ_{\omega}=\log f_{\vartheta}(\omega)d\omega+2\pi^{1/2}n^{-1/2}dW_{\omega
}\text{, }\omega\in\lbrack-\pi,\pi]\label{par-whitenoise}%
\end{equation}
and note that in the family $\left(  f_{\vartheta},\vartheta\in\Theta\right)
$, this model has the same asymptotic Fisher information. This is in agreement
with the LAN result for the spectral density model, but it suggests that the
above white noise approximation might also be true for larger (i.e.
nonparametric) spectral density classes $\Sigma$.

As a second piece of evidence for the white noise approximation in the
nonparametric case we take known results about the approximate spectral
decomposition of the Toeplitz covariance matrix $\Gamma_{n}(f)$. It is a
classical difficulty in time series analysis that the exact eigenvalues and
eigenvectors of $\Gamma_{n}(f)$ cannot easily be found and used for inference
about $f$; in particular, the eigenvectors depend on $f$. However for an
approximation which is a circulant matrix (denoted $\tilde{\Gamma}_{n}(f)$
below), the eigenvectors are independent of $f$ and the eigenvalues are
approximately $f(\omega_{j})$ where $\omega_{j}$ are the points of an
equispaced grid of size $n$ in $[-\pi,\pi]$. If the approximation by
$\tilde{\Gamma}_{n}(f)$ were justified, one could apply an orthogonal
transformation to the data $y^{(n)}$ and obtain a Gaussian scale model
\begin{equation}
z_{j}=f^{1/2}(\omega_{j})\xi_{j}\text{, }j=1,\ldots,n\label{Gauss-scale-model}%
\end{equation}
where $\xi_{j}$ are independent standard normal. For this model, nonparametric
asymptotic equivalence theory was developed in Grama and Nussbaum (1998).
Results there, for certain smoothness classes $f\in\Sigma$, with $f$ bounded
away from $0$, lead to the nonparametric version of the white noise model
(\ref{par-whitenoise})
\begin{equation}
dZ_{\omega}=\log f(\omega)d\omega+2\pi^{1/2}n^{-1/2}dW_{\omega}\text{, }%
\omega\in\lbrack-\pi,\pi]\text{, }f\in\Sigma.\label{np-whitenoise}%
\end{equation}
Our proof of asymptotic equivalence will in fact be based on the approximation
of the covariance matrix $\Gamma_{n}(f)$ by the circulant $\tilde{\Gamma}%
_{n}(f)$, cf. Brockwell and Davis (1991), \S \ 4.5. However we shall see that
this tool does not enable a staightforward approximation of the data $y^{(n)}$
in total variation or Hellinger distance. Therefore our argument for
asymptotic equivalence will be somewhat indirect, involving \textquotedblright
bracketing\textquotedblright\ of the experiment $\mathcal{E}_{n}$ by upper and
lower bounds (in the sense of informativity) and also a preliminary
localization of the parameter space.

To formulate our main result, define a parameter space $\Sigma$ of spectral
densities as follows. For $M>0$, define a set of real valued even functions on
$[-\pi,\pi]$
\[
\mathcal{F}_{M}=\left\{  f:M^{-1}\leq f(\omega)\text{, }f(\omega
)=f(-\omega),\omega\in\lbrack-\pi,\pi]\right\}  .
\]
Thus our spectral densities are assumed uniformly bounded away from $0.$ Let
$L_{2}(-\pi,\pi)$ be the usual (real) $L_{2}$-space on $[-\pi,\pi]$; for any
$f\in L_{2}(-\pi,\pi)$, let $\gamma_{f}(k)$, $k\in\mathbb{Z}$ be the Fourier
coefficients according to (\ref{first-equa}). For any $\alpha>0$ and $M>0$
let
\begin{equation}
W^{\alpha}(M)=\left\{  f\in L_{2}(-\pi,\pi):\gamma_{f}^{2}(0)+\sum_{k=-\infty
}^{\infty}|k|^{2\alpha}\gamma_{f}^{2}(k)\leq M\right\}  .\label{sobol-ball}%
\end{equation}
These sets correspond to balls in the periodic fractional Sobolev scale with
smoothness coefficient $\alpha$. Note that for $\alpha>1/2$, by an embedding
theorem (Lemma \ref{lem-sobol-embed}, Appendix), functions in $W^{\alpha}(M)$
are also uniformly bounded. Define an a priori set for given $\alpha>0,$
$M>0$
\[
\Sigma_{\alpha,M}=W^{\alpha}(M)\cap\mathcal{F}_{M}.
\]
Consider also a Gaussian scale model (\ref{Gauss-scale-model}) where the
values $f(\omega_{j})$ are replaced by local averages
\[
J_{j,n}\left(  f\right)  =n\int_{\left(  j-1\right)  /n}^{j/n}f(2\pi
x-\pi)dx,\text{ }j=1,\ldots,n
\]

\begin{theorem}
\label{theor-main-1}Let $\Sigma$ be a set of spectral densities contained in
$\Sigma_{\alpha,M}$ for some $M>0$ and $\alpha>1/2$. Then the experiments
given by observations
\begin{align*}
& y(1),\ldots,y(n)\text{,}\text{ a stationary centered Gaussian sequence with
spectral density }f\\
& z_{1},\ldots,z_{n}\text{, where }z_{j}\text{ are independent }%
N(0,J_{j,n}\left(  f\right)  )
\end{align*}
with $f\in\Sigma$ are asymptotically equivalent.
\end{theorem}

Let $\left\Vert \cdot\right\Vert _{B_{p,q}^{\alpha}}$ be the Besov norm on the
interval $[-\pi,\pi]$ with smoothness index $\alpha$ (see Appendix, Section
\ref{subsec-besov}). For the second main result we impose a smoothness
condition involving this norm for the $\alpha>1/2$ from above and $p=q=6$.

\begin{theorem}
\label{theor-main-2}Let $\Sigma$ be a set of spectral densities as in Theorem
(\ref{theor-main-1}), fulfilling additionally $\left\Vert f\right\Vert
_{B_{6,6}^{\alpha}}\leq M$ for all $f\in\Sigma$. Then the experiments given
respectively by observations
\begin{align*}
& z_{1},\ldots,z_{n}\text{, where }z_{j}\text{ are independent }%
N(0,J_{j,n}\left(  f\right)  )\\
& dZ_{\omega}=\log f(\omega)d\omega+2\pi^{1/2}n^{-1/2}dW_{\omega}\text{,
}\omega\in\lbrack-\pi,\pi]
\end{align*}
with $f\in\Sigma$ are asymptotically equivalent.
\end{theorem}

The proof of this result is in the thesis Zhou (2004). The present paper is
devoted to the proof of Theorem \ref{theor-main-1}.

In nonparametric asymptotic equivalence theory, some constructive results have
recently been obtained, i.e. explicit equivalence maps have been exhibited
which allow to carry over optimal decision function from one sequence of
experiments to the other. Brown and Low (1996) and Brown, Low and Zhang (2002)
obtained constructive results for white noise with drift and Gaussian
regression with nonrandom and random design. Brown, Carter, Low and Zhang
(2004) found such equivalence maps (Markov kernels) for the i.i.d. model on
the unit interval (density estimation) and the model of Gaussian white noise
with drift; cf. also Carter (2002). The theoretical (nonconstructive) variant
of this result had earlier been established in Nussbaum (1996), in the sense
of an existence proof for pertaining Markov kernels. This indirect approach
relied on the well known connection to likelihood processes of experiments,
cf. Le Cam and Yang (2000). In the present paper, the result of Theorem
\ref{theor-main-1} are of nonconstructive type, using a variety of methods for
bounding the $\Delta$-distance between the time series experiment and the
model of independent zero mean Gaussians. Similarly, the proof of Theorem
\ref{theor-main-2} in Zhou (2004) is nonconstructive, but it appears likely in
that a second step, relatively simple "workable" equivalence maps can be
found, at least for the case of Theorem \ref{theor-main-1} related to the
classical result about asymptotic independence of discrete Fourier transforms.

To further discuss the context of the main results, we note the following points.

\textbf{1.}\textit{\ Asymptotic independence of discrete Fourier transforms. }
Let%
\[
d_{n}(\omega)=\sum_{k=1}^{n}\exp\left(  -\mathbf{i}k\omega\right)
y(k),\omega\in(-\pi,\pi)
\]
be the discrete Fourier transform of the time series $y(1),\ldots,y(n)$.
Assume $n$ is uneven and let $\eta_{j}$ be complex standard normal variables.
It is well known that for the Fourier frequencies $\omega_{j}=2\pi j/n$,
$j=1,\ldots,(n-1)/2$ in $(0,\pi)$, there is an asymptotic distribution
\[
\left(  \pi n\right)  ^{-1/2}d_{n}(\omega_{j})\approx\exp(\mathbf{i}\omega
_{j})f^{1/2}(\omega_{j})\eta_{j}%
\]
and the values are asymptotically uncorrelated for distinct $\omega_{j}$,
$\omega_{k}$. For a precise formulation cf. relation
(\ref{cov-matrix-approx-elem}) below or Brockwell and Davis (1991),
Proposition 4.5.2. This fact is the basis for many inference methods (e.g.
Dahlhaus and Janas (1996)); see Lahiri (2003) for an extended discussion of
the asymptotic independence. A linear transformation to $n-1$ independent real
normals and adding a real normal according to $(2\pi n)^{-1/2}d_{n}(0)\approx
N(0,f(0))$ suggests the Gaussian scale model (\ref{Gauss-scale-model}).

\textbf{2.} \textit{Log-periodogram regression}. Consider also the
periodogram
\[
I_{n}(\omega)=\frac{1}{2\pi n}\left\vert d_{n}(\omega)\right\vert ^{2}.
\]
Note the equality in distribution $\left\vert \eta_{j}\right\vert ^{2}\sim
\chi_{2}^{2}\sim2e_{j}$, where $e_{j}$ is standard exponential. As a
consequence of the above result about $d_{n}(\omega_{j})$, we have for
$j=1,\ldots,(n-1)/2$
\begin{equation}
I_{n}(\omega_{j})\approx f(\omega_{j})e_{j}\label{periodogram-repres}%
\end{equation}
with asymptotic independence. Assuming this model exact and taking a logarithm
gives rise to the inference method of \textit{\ }log-periodogram regression
(for an account cf. Fan and Gijbels (1996), sec. 6.4)

\textbf{3. }\textit{The Whittle approximation. }This is\textit{\ }an
approximation to $-n^{-1}$ times the log-likelihood of the time series
$y(1),\ldots,y(n)$. In a parametric model $f_{\vartheta}$, $\vartheta\in
\Theta$, with multivariate normal law $N_{n}\left(  0,\Gamma_{n}(f_{\vartheta
})\right)  $, computation of the MLE involves inverting the covariance matrix
$\Gamma_{n}(f_{\vartheta})$, which is difficult since both eigenvectors and
eigenvalues depend on $\vartheta$ in general. Replacing $\Gamma_{n}%
^{-1}(f_{\vartheta})$ by $\Gamma_{n}(1/4\pi^{2}f_{\vartheta})$ and using an
approximation to $n^{-1}\log\Gamma_{n}(f_{\vartheta})$ leads to an expression
$L^{W}(f)+\log2\pi$ where
\begin{equation}
L^{W}(f)=\frac{1}{4\pi}\int_{-\pi}^{\pi}\left(  \log f_{\vartheta}%
(\omega)+\frac{I_{n}(\omega)}{f_{\vartheta}(\omega)}\right)  d\omega
\label{whittle-func}%
\end{equation}
is the Whittle likelihood (cf. Dahlhaus (1988) for a brief exposition and
references). A closely related expression is obtained by assuming the model
(\ref{periodogram-repres}) exact: then $-n^{-1}$ times the log-likelihood is
\[
L_{n}^{W}(f)=n^{-1}%
{\displaystyle\sum\limits_{j=1}^{(n-1)/2}}
\left(  \log f_{\vartheta}(\omega_{j})+\frac{I_{n}(\omega_{j})}{f_{\vartheta
}(\omega_{j})}\right)
\]
i. e. a discrete approximation to (\ref{whittle-func}). For applications of
the Whittle likelihood to nonparametric inference cf. Dahlhaus and Polonik (2002).

\textbf{4. }\textit{Asymptotics for} $L^{W}(f)$. The accuracy of the Whittle
approximation has been described as follows (Coursol and Dacunha-Castelle
(1982), Dzhaparidze (1986), Theorem 1, p. 52) . Let $L_{n}(f)$ be the
log-likelihood in the experiment (\ref{exper-1}); then
\begin{equation}
L_{n}(f)=-nL^{W}(f)-n\log2\pi+O_{P}(1)\label{whittle-approx-accuracy}%
\end{equation}
uniformly over $f\in$ $\Sigma_{1/2,M}$. This justifies use of $L^{W}(f)$ as a
contrast function, e.g. it yields asymptotic efficiency of the Whittle MLE in
parametric models (Dzhaparidze (1986), Chap. II), but falls short of providing
asymptotic equivalence in the Le Cam sense. Indeed if
(\ref{whittle-approx-accuracy}) were true with $o_{P}(1)$ in place of
$O_{P}(1)$ and with $L^{W}(f)$ replaced by $L_{n}^{W}(f)$ then this would
already imply total variation equivalence, up to an orthogonal transform, of
the exact model (\ref{periodogram-repres}) with $f\in$ $\Sigma_{1/2,M}$ (via
the Scheffe lemma argument of Delattre and Hoffmann (2002)). In section 2
below (cf. relation (\ref{fails})) we note a corresponding negative result,
essentially that this total variation approximation over $f\in$ $\Sigma
_{1/2,M}$ does not take place.

\textbf{5.} \textit{Conditions for Theorem \ref{theor-main-2}}. For a narrower
parameter space, i. e. a H\"{o}lder ball with smoothness index $\alpha>1/2,$
the result of Theorem \ref{theor-main-2} has been proved by Grama and Nussbaum
(1998). Note that the Sobolev balls $W^{\alpha}(M)$ figuring in Theorem
\ref{theor-main-1} are natural parameter sets of spectral densities since the
smoothness condition is directly stated in terms of the autocovariance
function $\gamma_{f}(\cdot)$. The Besov balls $B_{p,p}^{\alpha}(M)$ given in
terms of the norm $\left\Vert \cdot\right\Vert _{B_{p,p}^{\alpha}}$are
intermediate between $L_{2}$-Sobolev and H\"{o}lder balls. For the white noise
approximation of the i.i.d. (density estimation) model, Brown, Carter, Low and
Zhang (2004) succeeded in weakening the H\"{o}lder ball condition in Nussbaum
(1996) to a condition that $\Sigma$ is compact both in the Besov spaces
$B_{2,2}^{1/2}$ and $B_{4,4}^{1/2}$ on the unit interval. This is immediately
implied by $\Sigma\subset B_{4,4}^{\alpha}(M)$ for some $\alpha>1/2$. Our
condition for Theorem \ref{theor-main-2} is slightly stronger, i.e.
$\Sigma\subset B_{6,6}^{\alpha}(M)$ for some $\alpha>1/2$. In Remark
\ref{rem-periodic-besov} (Appendix) we note a sufficient condition in terms of
the autocovariance function $\gamma_{f}(\cdot)$, i.e. give a description of
the periodic version of the Besov ball.

Throughout this paper we adopt the notation that $C$ represents a constant
independent of $n$ and the parameter (spectral density) $f\in\Sigma$, and the
value of which may change at each occurrence, even on the same line.

\paragraph{\textbf{Relations between experiments. }}

All measurable sample spaces are assumed to be Polish (complete separable)
metric spaces equipped with their Borel sigma algebra. For measures $P,$ $Q$
on the same sample space, let $\left\Vert P-Q\right\Vert _{TV}$ be the total
variation distance. For the general case where $P,$ $Q$ are not necessarily on
the same sample space, suppose $K$ is a Markov kernel such that $KP$ is a
measure on the same sample space as $Q$. In that case, $\left\Vert
Q-KP\right\Vert _{TV}$ is defined and will be used as generic notation for a
Markov kernel $K$.

Consider now experiments (families of measures) $\mathcal{F}=\left(
Q_{f},\;f\in\Sigma\right)  $ and $\mathcal{E}=\left(  P_{f},\;f\in
\Sigma\right)  $, with the same parameter space $\Sigma$. All experiments here
are assumed dominated by a sigma-finite measure on their respective sample
space. If $\mathcal{E}$ and $\mathcal{F}$ are on the same sample space, define
their total variation distance
\[
\Delta_{0}\left(  \mathcal{E},\mathcal{F}\right)  =\sup_{f\in\Sigma}\left\Vert
Q_{f}-P_{f}\right\Vert _{TV}.
\]
In the general case, the deficiency of $\mathcal{E}$ with respect to
$\mathcal{F}$ is defined as
\[
\delta\left(  \mathcal{E},\mathcal{F}\right)  =\inf_{K}\sup_{f\in\Sigma
}\left\Vert Q_{f}-KP_{f}\right\Vert _{TV}%
\]
where $\inf$ extends over all appropriate Markov kernels. Le Cam's
pseudodistance $\Delta\left(  \mathcal{\cdot},\cdot\right)  $ between
$\mathcal{E}$ and $\mathcal{F}$ then is
\[
\Delta\left(  \mathcal{E},\mathcal{F}\right)  =\max\left(  \delta\left(
\mathcal{E},\mathcal{F}\right)  ,\delta\left(  \mathcal{F},\mathcal{E}\right)
\right)  .
\]
Furthermore, we will use the following notation involving experiments
$\mathcal{E},\mathcal{F}$ or sequences of such $\mathcal{E}_{n}=\left(
P_{n,f},\;f\in\Sigma\right)  $ and $\mathcal{F}_{n}=\left(  Q_{n,f}%
,\;f\in\Sigma\right)  $.

\textbf{Notation.}%

\begin{tabular}
[t]{lllll}%
$\mathcal{E}$ & $\mathbf{\preceq}$ & $\mathcal{F}$ & ($\mathcal{F}$ more
informative than $\mathcal{E}$): & $\delta\left(  \mathcal{F},\mathcal{E}%
\right)  =0$\\
$\mathcal{E}$ & $\mathcal{\sim}$ & $\mathcal{F}$ & (equivalent): &
$\Delta\left(  \mathcal{E},\mathcal{F}\right)  =0$\\
$\mathcal{E}_{n}$ & $\simeq$ & $\mathcal{F}_{n}$ & (asymptotically total
variation equivalent): & $\Delta_{0}\left(  \mathcal{F}_{n},\mathcal{E}%
_{n}\right)  \rightarrow0$\\
$\mathcal{E}_{n}$ & $\precsim$ & $\mathcal{F}_{n}$ & ($\mathcal{F}_{n}$
asymptotically more informative than $\mathcal{E}_{n}$): & $\delta\left(
\mathcal{F}_{n},\mathcal{E}_{n}\right)  \rightarrow0$\\
$\mathcal{E}_{n}$ & $\approx$ & $\mathcal{F}_{n}$ & (asymptotically
equivalent): & $\Delta\left(  \mathcal{F}_{n},\mathcal{E}_{n}\right)
\rightarrow0$%
\end{tabular}

Note that "more informative" above is used in the sense of a semi-ordering,
i.e. its actual meaning is "at least as informative". We shall also write the
relation $\simeq$ in a less formal way between data vectors such as
$x^{(n)}\simeq y^{(n)}$, if it is clear from the context which experiments the
data vectors represent.

\section{The periodic Gaussian experiment}

From now on we shall assume that $n$ is uneven. Our argument for asymptotic
equivalence is such that it easily allows extension to the case of general
sequences $n\rightarrow\infty$ (cf. Remark \ref{rem-even} for details).

Recall that the covariance matrix $\Gamma_{n}=\Gamma_{n}(f)$ has the Toeplitz
form $(\Gamma_{n})_{j,k}=\gamma(k-j)$, $j,k=1,\ldots,n,$ i.e.
\[
\Gamma_{n}=\left(
\begin{array}
[c]{ccccc}%
\gamma(0) & \gamma(1) & \ldots & \gamma(n-2) & \gamma(n-1)\\
\gamma(1) & \gamma(0) & \ldots & \ldots & \gamma(n-2)\\
\ldots & \ldots & \ldots & \ldots & \ldots\\
\gamma(n-2) & \ldots & \ldots & \gamma(0) & \gamma(1)\\
\gamma(n-1) & \gamma(n-2) & \ldots & \gamma(1) & \gamma(0)
\end{array}
\right)  .
\]
Following Brockwell and Davis (1991), \S \ 4.5 we shall define a circulant
matrix approximation by
\[
\tilde{\Gamma}_{n}=\left(
\begin{array}
[c]{ccccc}%
\gamma(0) & \gamma(1) & \ldots & \gamma(2) & \gamma(1)\\
\gamma(1) & \gamma(0) & \ldots & \ldots & \gamma(2)\\
\ldots & \ldots & \ldots & \ldots & \ldots\\
\gamma(2) & \ldots & \ldots & \gamma(0) & \gamma(1)\\
\gamma(1) & \gamma(2) & \ldots & \gamma(1) & \gamma(0)
\end{array}
\right)
\]
where in the first row, the central element and the one following it coincide
with $\gamma((n-1)/2)$. More precisely, for given uneven $n$ define a function
on integers $h$ with $\left\vert h\right\vert <n$
\[
\tilde{\gamma}_{(n),f}(h)=\left\{
\begin{array}
[c]{c}%
\gamma_{f}(h)\text{, }\left\vert h\right\vert \leq(n-1)/2\\
\gamma_{f}(n-\left\vert h\right\vert )\text{, }(n+1)/2\leq\left\vert
h\right\vert \leq n-1
\end{array}
\right.
\]
and set
\begin{equation}
(\tilde{\Gamma}_{n})_{j,k}(f)=\tilde{\gamma}_{(n),f}(k-j),j,k=1,\ldots
,n.\label{gamma-n-tilde-def}%
\end{equation}
We shall also write $\tilde{\Gamma}_{n}(f)$ for the corresponding $n\times n$
matrix, or simply $\tilde{\Gamma}_{n}$ and $\tilde{\gamma}_{(n)}(h)$ if the
dependence on $f$ is understood. Define
\begin{equation}
\omega_{j}=\frac{2\pi j}{n}\text{, }\left\vert j\right\vert \leq
(n-1)/2.\label{omega-j-def}%
\end{equation}
It is well known (see Brockwell and Davis (1991), relation 4.5.5) that the
spectral decomposition of $\tilde{\Gamma}_{n}$ can be described as follows. We
have
\begin{equation}
\tilde{\Gamma}_{n}=\sum_{\left\vert j\right\vert \leq(n-1)/2}\lambda
_{j}\mathbf{u}_{j}\mathbf{u}_{j}^{\prime}\label{spec-decomp-gamma-tilde}%
\end{equation}
where $\lambda_{j}$ are real eigenvalues and $\mathbf{u}_{j}$ are real
orthonormal eigenvectors. The eigenvalues are
\[
\lambda_{j}=\sum_{\left\vert k\right\vert \leq(n-1)/2}\gamma(k)\exp
(-\mathrm{i}\omega_{j}k),\;\left\vert j\right\vert \leq(n-1)/2.
\]
Note that $\lambda_{j}=$ $\lambda_{-j}$, $j\neq0$ and that the $\lambda_{j}$
are approximate values of $2\pi f$ in the points $\omega_{j}.$ Indeed define
\begin{equation}
\tilde{f}_{n}(\omega)=\frac{1}{2\pi}\sum_{\left\vert k\right\vert \leq
(n-1)/2}\gamma(k)\mathrm{\exp}{\normalsize (\mathrm{i}k\omega),\;\omega
\in\lbrack-\pi,\pi]}\label{f-n-bar}%
\end{equation}
a truncated Fourier series approximation to $f$; then $\tilde{f}_{n}$ is an
even function on $[-\pi,\pi]$ and
\begin{equation}
\lambda_{j}=2\pi\tilde{f}_{n}(\omega_{j}),\;\left\vert j\right\vert
\leq(n-1)/2.\label{ev-as-funcvalues}%
\end{equation}
The eigenvectors are
\begin{align}
\;\mathbf{u}_{0}^{\prime}  & =n^{-1/2}\left(  1,\ldots,1\right)
,\label{ev-1}\\
\mathbf{u}_{j}^{\prime}  & =(2/n)^{1/2}\left(  1,\cos(\omega_{j}),\cos
(2\omega_{j})\ldots,\cos((n-1)\omega_{j})\right)  ,\;\label{ev-2}\\
\mathbf{u}_{-j}^{\prime}  & =(2/n)^{1/2}\left(  0,\sin(\omega_{j}%
),\sin(2\omega_{j})\ldots,\sin((n-1)\omega_{j})\right)  ,\;j=1,\ldots
,(n-1)/2.\label{ev-3}%
\end{align}
In our setting, the circulant matrix $\tilde{\Gamma}_{n}$ is positive definite
for $n$ large enough. Indeed, Lemma \ref{lem-sobol-embed} Appendix implies
that $\tilde{f}_{n}\geq M^{-1}/2$ uniformly over $f\in\Sigma$, for $n$ large
enough, so that $\tilde{\Gamma}_{n}(f)$ is a covariance matrix. Define the
experiment, in analogy to (\ref{exper-1}),
\begin{equation}
\mathcal{\tilde{E}}_{n}=\left(  N_{n}(0,\tilde{\Gamma}_{n}(f)),f\in
\Sigma\right)  \label{periodic-exper}%
\end{equation}
with data $\tilde{y}^{(n)}$, say. The sequence $\tilde{y}^{(n)}$ may be called
a \textquotedblright periodic process\textquotedblright\ since it can be
represented in terms of independent standard Gaussians $\xi_{j}$, as a finite
sum
\begin{equation}
\tilde{y}^{(n)}=\sum_{\left\vert j\right\vert \leq(n-1)/2}\lambda_{j}%
^{1/2}\mathbf{u}_{j}\xi_{j}\label{specrep-periodic}%
\end{equation}
where the vector $\mathbf{u}_{j}$ describes a deterministic oscillation (cp.
(\ref{ev-1})-(\ref{ev-3})). Accordingly $\mathcal{\tilde{E}}_{n}$ will be
called a \textit{periodic Gaussian experiment}.

The periodic process $\tilde{y}^{(n)}$ is known to approximate the original
time series $y^{(n)}$ in the following sense. Define the $n\times n$-matrix
\begin{equation}
U_{n}=\left(  \mathbf{u}_{-(n-1)/2,}\ldots,\mathbf{u}_{(n-1)/2}\right)
\label{orthog-transform-equiv-map}%
\end{equation}
and consider the transforms
\[
z^{(n)}=(2\pi)^{-1/2}U_{n}^{\prime}y^{(n)},\;\tilde{z}^{(n)}=(2\pi
)^{-1/2}U_{n}^{\prime}\tilde{y}^{(n)}.
\]
Denote $\mathrm{Cov}(z^{(n)})$ the covariance matrix of the random vector
$z^{(n)}$. Then we have (Brockwell and Davis (1991), Proposition 4.5.2), for
given $f\in\Sigma$
\begin{equation}
\sup_{1\leq i,j\leq n}\left\vert \mathrm{Cov}(z^{(n)})_{i,j}-\mathrm{Cov}%
(\tilde{z}^{(n)})_{i,j}\right\vert \rightarrow0\text{ as }n\rightarrow
\infty\text{. }\label{cov-matrix-approx-elem}%
\end{equation}
Since $\mathrm{Cov}(\tilde{z}^{(n)})$ is diagonal with diagonal elements
$\lambda_{j}/2\pi$, this means that the elements of $z^{(n)}$ are
approximately uncorrelated for large $n$.

Note that $\tilde{z}^{(n)}$ can also be written, in accordance with
(\ref{specrep-periodic}) and (\ref{ev-as-funcvalues})
\begin{equation}
\tilde{z}^{(n)}=\left(  \tilde{f}_{n}^{1/2}(\omega_{j})\xi_{j}\right)
_{\left\vert j\right\vert \leq(n-1)/2}\label{n-vec-of-ind}%
\end{equation}
which is nearly identical with the Gaussian scale model
(\ref{Gauss-scale-model}). Thus the question appears whether the approximation
(\ref{cov-matrix-approx-elem}) can be strengthened to a total variation
approximation of the respective laws $\mathcal{L}\left(  z^{(n)}|f\right)  $
and $\mathcal{L}\left(  \tilde{z}^{(n)}|f\right)  $.

The answer to that is negative; let us introduce some notation. For $n\times
n$ matrices $A=(a_{jk})$ define the Euclidean norm $\left\|  A\right\|  $ by
\[
\left\|  A\right\|  ^{2}:=\mathrm{tr}\left[  A^{\prime}A\right]  =\sum
_{j=1}^{n}\sum_{k=1}^{n}a_{jk}^{2}.
\]
If $A$ is symmetric, we denote the largest and smallest eigenvalues by
$\lambda_{\max}(A)$, $\lambda_{\min}(A)$. For later use, we also define the
operator norm of (not necessarily symmetric) $A$ by
\[
\left|  A\right|  :=\left(  \lambda_{\max}(A^{\prime}A)\right)  ^{1/2}.
\]
If A is symmetric nonnegative definite then $\left|  A\right|  =\lambda_{\max
}(A)$. The following lemma shows that the Hellinger distance between the laws
of $y^{(n)}$ and $\tilde{y}^{(n)}$ depends crucially on the total Euclidean
distance $\left\|  \Gamma_{n}(f)-\tilde{\Gamma}_{n}(f)\right\|  $ between the
covariance matrices, so that an elementwise convergence as in
(\ref{cov-matrix-approx-elem}) is not enough.

\begin{lemma}
\label{lem-helldist-covmatr} Let $A,B$ be $n\times n$ covariance matrices and
suppose that for some $M>1$
\[
0<M^{-1}\leq\lambda_{\min}(A)\;\text{and }\lambda_{\max}(A)\leq M.
\]
Then there exist $\epsilon=\epsilon_{M}>0$ and $K=K_{M}>1$ not depending on
$A,B$ and $n$ such that $\left\Vert A-B\right\Vert \leq\epsilon$ implies
\[
K^{-1}\;\left\Vert A-B\right\Vert ^{2}\leq H^{2}\left(  N_{n}(0,A),N_{n}%
(0,B)\right)  \leq K\;\left\Vert A-B\right\Vert ^{2}.
\]
where $H(\cdot,\cdot)$ is the Hellinger distance.
\end{lemma}

The proof is in section \ref{sec-proofs}. To apply this lemma, set
$A=\Gamma_{n}(f)$, $B=\tilde{\Gamma}_{n}(f)$ and note that, since $f\in\Sigma$
is bounded and bounded away from $0$ (both uniformly over $f\in\Sigma$), the
condition on the eigenvalues of $\Gamma_{n}(f)$ is fulfilled, also uniformly
over $f\in\Sigma$ (Brockwell and Davis (1991), Proposition 4.5.3). We shall
see that the expression $\left\Vert \Gamma_{n}(f)-\tilde{\Gamma}%
_{n}(f)\right\Vert ^{2}$ is closely related to a Sobolev type seminorm for
smoothness index $1/2$. For any $f\in L_{2}(-\pi,\pi)$ given by
(\ref{first-equa}) set
\begin{equation}
\left\vert f\right\vert _{2,\alpha}^{2}:=\sum_{k=-\infty}^{\infty}\left\vert
k\right\vert ^{2\alpha}\gamma_{f}^{2}(k),\;\left\Vert f\right\Vert _{2,\alpha
}^{2}:=\gamma_{f}^{2}(0)+\left\vert f\right\vert _{2,\alpha}^{2}%
\label{normdef}%
\end{equation}
provided the right side is finite; the Sobolev ball $W^{\alpha}(M)$ given by
(\ref{sobol-ball}) is then described by $\left\Vert f\right\Vert _{2,\alpha
}^{2}\leq M$. Also, for any natural $m$ define a finite dimensional linear
subspace of $L_{2}(-\pi,\pi)$
\[
L_{m}=\left\{  f\in L_{2}(-\pi,\pi):\int f(\omega)\mathrm{\exp}%
{\normalsize (\mathrm{i}k\omega)d\omega}=0\text{, }\left\vert k\right\vert
>m\right\}  .
\]

\begin{lemma}
\label{lem-covmatnorm-sobol}\textbf{(i)} For any $f\in\Sigma$ we have
\begin{equation}
\left\Vert \Gamma_{n}(f)-\tilde{\Gamma}_{n}(f)\right\Vert ^{2}\leq2\left\vert
f\right\vert _{2,1/2}^{2}\label{upperbound-periodic-orig}%
\end{equation}
and for $f\in\Sigma\cap L_{(n-1)/2}$
\[
\left\vert f\right\vert _{2,1/2}^{2}=\left\Vert \Gamma_{n}(f)-\tilde{\Gamma
}_{n}(f)\right\Vert ^{2}.
\]
\textbf{(ii)} For any $f,f_{0}\in\Sigma$ we have%
\begin{equation}
\left\Vert \Gamma_{n}(f)-\Gamma_{n}(f_{0})-\left(  \tilde{\Gamma}%
_{n}(f)-\tilde{\Gamma}_{n}(f_{0})\right)  \right\Vert ^{2}\leq2\left\vert
f-f_{0}\right\vert _{2,1/2}^{2}.\label{upperbound-periodic-local}%
\end{equation}

\end{lemma}

\begin{proof}
\textbf{\ (i) }From the definition of $\Gamma_{n}(f)$ and $\tilde{\Gamma}%
_{n}(f)$ in terms of $\gamma(\cdot),\tilde{\gamma}_{(n)}(\cdot)$ we
immediately obtain
\[
\left\Vert \Gamma_{n}(f)-\tilde{\Gamma}_{n}(f)\right\Vert ^{2}=\sum
_{\left\vert k\right\vert \leq n-1}(n-\left\vert k\right\vert )\left(
\gamma(k)-\tilde{\gamma}_{(n)}(k)\right)  ^{2}%
\]%
\begin{align}
& =\sum_{\left\vert k\right\vert =(n+1)/2}^{n-1}(n-\left\vert k\right\vert
)\left(  \gamma(k)-\gamma(n-|k|)\right)  ^{2}=2\sum_{k=1}^{(n-1)/2}k\left(
\gamma(k)-\gamma(n-k)\right)  ^{2}\label{covnorm-expr-1}\\
& \leq2\sum_{k=1}^{(n-1)/2}2k\left(  \gamma^{2}(k)+\gamma^{2}(n-k)\right)
\leq4\sum_{k=1}^{n-1}k\gamma^{2}(k)\leq2\left\vert f\right\vert _{2,1/2}%
^{2}.\nonumber
\end{align}
The first inequality is proved. The second one follows immediately from
(\ref{covnorm-expr-1}).

\textbf{(ii)} Note that for any $n$, the mapping $f\rightarrow\Gamma_{n}(f) $
if it is defined by (\ref{gamma-n-def}) for any $f\in L_{2}(-\pi,\pi)$ is
linear, and the same is true for $f\rightarrow\tilde{\Gamma}_{n}(f)$ defined
by (\ref{gamma-n-tilde-def}). Hence
\[
\Gamma_{n}(f)-\Gamma_{n}(f_{0})=\Gamma_{n}(f-f_{0}),\;\tilde{\Gamma}%
_{n}(f)-\tilde{\Gamma}_{n}(f_{0})=\tilde{\Gamma}_{n}(f-f_{0}).
\]
Now the argument is completely analogous to (i) if $\gamma(k)=\gamma_{f}(k) $
is replaced by $\gamma_{f-f_{0}}(k)$.
\end{proof}

Our assumption $f\in\Sigma,$ i.e. $\left\Vert f\right\Vert _{2,\alpha}^{2}\leq
M$ for some $\alpha>1/2$ provides an upper bound $M$ for $\left\vert
f\right\vert _{2,1/2}^{2}$ but does guarantee that this term is uniformly
small. Thus we are not able to utilize Lemma \ref{lem-helldist-covmatr} to
approximate $\mathcal{E}_{n}$ by $\mathcal{\tilde{E}}_{n}$ in Hellinger
distance. In fact this Hellinger distance approximation does not take place:
take a fixed $m$, select $f\in\Sigma\cap L_{m}$ such that $\left\Vert
f\right\Vert _{2,1/2}^{2}<\epsilon$ with $\epsilon$ from Lemma
\ref{lem-helldist-covmatr} and use the lower bound in this lemma to show that
\begin{equation}
H^{2}\left(  N_{n}(0,\Gamma_{n}(f)),N_{n}(0,\tilde{\Gamma}_{n}(f))\right)
\geq K^{-1}\epsilon^{2}\label{fails}%
\end{equation}
for all sufficiently large $n$. Thus the direct approximation of the time
series data $y^{(n)}$ by the periodic process $\tilde{y}^{(n)}$ in total
variation distance fails.

However that does not contradict asymptotic equivalence since the latter
allows for a randomization mapping (Markov kernel) applied to $\tilde{y}%
^{(n)}$ and $y^{(n)}$, respectively, before total variation distance of the
laws is taken. We will show the existence of appropriate Markov kernels in an
indirect way, via a bracketing of the original time series experiment by upper
and lower bounds in the sense of informativity.

Let now $\mathcal{E}_{n}$ again be the time series experiment (\ref{exper-1});
we shall find an asymptotic bracketing, i.e. two sequences $\mathcal{\mathring
{E}}_{l,n}$, $\mathcal{\mathring{E}}_{u,n}$ such that
\[
\mathcal{\mathring{E}}_{l,n}\precsim\mathcal{E}_{n}\precsim\mathcal{\mathring
{E}}_{u,n}%
\]
and such that both $\mathcal{\mathring{E}}_{l,n}$ and $\mathcal{\mathring{E}%
}_{u,n}$ are asymptotically equivalent to $\mathcal{\tilde{E}}_{n}$ given by
(\ref{periodic-exper}), and to $\mathcal{\mathring{E}}_{n}$ representing the
independent Gaussians $z_{1},\ldots,z_{n}$ in Theorem \ref{theor-main-1}.

\section{Upper informativity bracket}

The spectral representation (\ref{specrep-periodic}) of the periodic sequence
$\tilde{y}^{(n)}=\left(  \tilde{y}(1),\ldots,\tilde{y}(n)\right)  ^{\prime}$
can be written
\[
\tilde{y}(t)=(2\pi/n)^{1/2}\tilde{f}_{n}^{1/2}(0)\xi_{0}+2(\pi/n)^{1/2}%
\sum_{j=1}^{(n-1)/2}\tilde{f}_{n}^{1/2}(\omega_{j})\cos((t-1)\omega_{j}%
)\xi_{j}%
\]%
\begin{equation}
+2(\pi/n)^{1/2}\sum_{j=-(n-1)/2}^{1}\tilde{f}_{n}^{1/2}(\omega_{j}%
)\sin((t-1)\omega_{j})\xi_{j},t=1,\ldots,n.\label{specrep-per-2}%
\end{equation}
We saw that here $\tilde{y}^{(n)}$ is a one-to-one function $\tilde{y}%
^{(n)}=U\tilde{z}^{(n)}$ of the $n$-vector of independent Gaussians $\tilde
{z}^{(n)}$ (cf. (\ref{n-vec-of-ind})), but the approximation of $\tilde
{y}^{(n)}$ to $y^{(n)}$ is not in the total variation sense (cf.
(\ref{fails})). Now take a limit in (\ref{specrep-per-2}) for $n\rightarrow
\infty$ and fixed $t$ and observe that (heuristically) this yields the
spectral representation of the original stationary sequence $y(t)$
\begin{equation}
y(t+1)=\int_{[0,\pi]}\sqrt{2}f^{1/2}(\omega)\cos(t\omega)dB_{\omega}%
+\int_{[-\pi,0]}\sqrt{2}f^{1/2}(\omega)\sin(t\omega)dB_{\omega}\text{,
}t=0,1,\ldots\label{specrep-class}%
\end{equation}
where $dB_{\omega}$ is standard Gaussian white noise on $[-\pi,\pi]$ (cf.
Brockwell and Davis (1991), Probl. 4.31). Here for any $n$, the vector
$y^{(n)}=\left(  y(1),\ldots,y(n)\right)  ^{\prime}$ is represented as a
functional of the continuous time process
\[
dZ_{\omega}^{\ast}=f^{1/2}(\omega)dB_{\omega},\omega\in\lbrack-\pi,\pi].
\]
Thus a completely observed process $Z_{\omega}^{\ast}$, $\omega\in\lbrack
-\pi,\pi]$ would represent an upper informativity bracket for any sample size
$n$, but this experiment is statistically trivial since the observation here
identifies the parameter $f$.

Our approach now is to construct an intermediate series $\tilde{y}^{(m,n)}$ of
size $n$ in which the uniform size $n$ grid of points $\omega_{j},\left\vert
j\right\vert \leq(n-1)/2$ is replaced by a finer uniform grid of $m>n$ points
in the representation (\ref{specrep-per-2}). Thus $\tilde{y}^{(n,m)}$ is a
functional not of $n$ independent Gaussians but of $m>n$ of these; call their
vector $\tilde{z}^{(m)}$. The random vector $\tilde{z}^{(m)}$ now represents
an upper informativity bracket which remains nontrivial (asymptotically) if
$m-n\rightarrow\infty$ not too quickly. An equivalent description of that idea
is as follows. Consider $m>n$ and the periodic process $\tilde{y}^{(m)}$ given
by (\ref{specrep-periodic}) where the original sample size $n$ is replaced by
$m$. Then define $\tilde{y}^{(n,m)}$ as the vector of the first $n$ components
of $\tilde{y}^{(m)}$. The law of $\tilde{y}^{(n,m)}$ is $N_{n}(0,\tilde
{\Gamma}_{n,m}(f))$ where $\tilde{\Gamma}_{n,m}(f)$ is the upper left $n\times
n$ submatrix of $\tilde{\Gamma}_{m}(f)$.

We now easily observe the improved approximation quality of $\tilde{y}%
^{(n,m)}$ for $y^{(n)}$. Assume that $m$ is also uneven. First note that for
$(m+1)/2\geq n$ we already obtain $\tilde{\Gamma}_{n,m}(f)=\Gamma_{n}(f)$.
This follows immediately from the definition of the circular matrix
$\tilde{\Gamma}_{m}(f)$ via the autocovariance function $\tilde{\gamma}%
_{(m)}(\cdot) $. However we would like to limit the increase of sample size,
i.e. require $m/n\rightarrow1$; therefore, in what follows we assume $m<2n-1$.

\begin{lemma}
\label{lem-upp-inf-brack}Assume $m$ is uneven, $n<m<2n-1$. Then for any
$f\in\Sigma$ we have
\[
\left\Vert \Gamma_{n}(f)-\tilde{\Gamma}_{n,m}(f)\right\Vert ^{2}\leq4\left(
m-n+1\right)  ^{1-2\alpha}\left\vert f\right\vert _{2,\alpha}^{2},
\]
and hence if $m=m_{n}$ is such that $m-n\rightarrow\infty$ as $n\rightarrow
\infty$ then
\begin{equation}
\sup_{f\in\Sigma}H^{2}\left(  N_{n}(0,\Gamma_{n}(f)),N_{n}(0,\tilde{\Gamma
}_{n,m}(f))\right)  \rightarrow0.\label{helldist-upperbrack}%
\end{equation}

\end{lemma}

\begin{proof}
From the definition of $\Gamma_{n}(f)$ and $\tilde{\Gamma}_{n,m}(f)$ we
immediately obtain
\[
\left\Vert \Gamma_{n}(f)-\tilde{\Gamma}_{n,m}(f)\right\Vert ^{2}%
=\sum_{\left\vert k\right\vert \leq n-1}(n-\left\vert k\right\vert )\left(
\gamma(k)-\tilde{\gamma}_{(m)}(k)\right)  ^{2}%
\]%
\[
=2\sum_{k=(m+1)/2}^{n-1}(n-k)\left(  \gamma(k)-\gamma(m-k)\right)  ^{2}%
\leq4\sum_{k=(m+1)/2}^{n-1}(n-k)\left(  \gamma^{2}(k)+\gamma^{2}(m-k)\right)
.
\]
Now note that for $m>n$, the relation $(m+1)/2\leq k\leq n-1$ implies
$k\geq(n+1)/2$ and therefore $n-k<k$, and note also $n-k<m-k$. We obtain an
upper bound
\begin{align*}
& \leq4\sum_{k=(m+1)/2}^{n-1}k\gamma^{2}(k)+4\sum_{k=(m+1)/2}^{n-1}%
(m-k)\gamma^{2}(m-k)\\
& =4\sum_{k=(m+1)/2}^{n-1}k\gamma^{2}(k)+4\sum_{k=m-n+1}^{(m-1)/2}k\gamma
^{2}(k)=4\sum_{k=m-n+1}^{n-1}k\gamma^{2}(k)\\
& \leq4(m-n+1)^{1-2\alpha}\sum_{k=m-n+1}^{n-1}k^{2\alpha}\gamma^{2}%
(k)\leq4(m-n+1)^{1-2\alpha}\left\vert f\right\vert _{2,\alpha}^{2}%
\end{align*}
where $\alpha>1/2$. This proves the first relation. For the second, recall
that $\left\vert f\right\vert _{2,\alpha}^{2}\leq M$ for $f\in\Sigma$ and
invoke Lemma \ref{lem-helldist-covmatr} together with the subsequent remark on
the eigenvalues of $\Gamma_{n}(f)$.
\end{proof}

Define the experiment
\[
\mathcal{\tilde{E}}_{n,m}=\left(  N_{n}(0,\tilde{\Gamma}_{n,m}(f)),f\in
\Sigma\right)
\]
then (\ref{helldist-upperbrack}) implies $\mathcal{E}_{n}\simeq\mathcal{\tilde
{E}}_{n,m}$ if $m-n\rightarrow\infty$. Moreover, we have $\mathcal{\tilde{E}%
}_{n,m}\preceq\mathcal{\tilde{E}}_{m}$ by definition, thus
\[
\mathcal{E}_{n}\precsim\mathcal{\tilde{E}}_{m}%
\]
in case $m-n\rightarrow\infty$. We know that $\mathcal{\tilde{E}}_{m}$ is
equivalent (via the linear transformation $(2\pi)^{-1/2}U^{\prime}$) to
observing data $\tilde{z}^{(n)}$ given by (\ref{n-vec-of-ind}). Define
$\mathcal{\mathring{E}}_{n}$ by
\begin{equation}
\mathcal{\mathring{E}}_{n}=\left(  N_{n}(0,\mathring{\Gamma}_{n}%
(f)),f\in\Sigma\right) \label{exp-e-dot-def}%
\end{equation}
where
\[
\mathring{\Gamma}_{n}(f)=Diag\left(  J_{j,n}(f)\right)  _{j=1,\ldots,n}.
\]
Note that the data $z_{1},\ldots,z_{n}$ in Theorem \ref{theor-main-1} are
represented by $\mathcal{\mathring{E}}_{n}$. We shall also write $\mathring
{z}^{(n)}$ for their vector, so that $\mathcal{L}(\mathring{z}^{(n)}%
|f)=N_{n}(0,\mathring{\Gamma}_{n}(f))$.

\begin{proposition}
\label{prop-explicit-equiv-map}We have $\mathcal{\mathring{E}}_{n}%
\approx\mathcal{\tilde{E}}_{n},$ with corresponding equivalence maps (Markov
kernels) as follows. Let $\tilde{y}^{(n)}$ and $\mathring{z}^{(n)}$ be data in
$\mathcal{\tilde{E}}_{n}$ and $\mathcal{\mathring{E}}_{n}$ respectively. Then,
for the orthogonal matrix $U_{n}$ given by (\ref{orthog-transform-equiv-map})%
\[
(2\pi)^{-1/2}U_{n}^{\prime}\tilde{y}^{(n)}\simeq\mathring{z}^{(n)}\text{, and
}(2\pi)^{1/2}U_{n}\mathring{z}^{(n)}\simeq\tilde{y}^{(n)}.
\]

\end{proposition}

\begin{proof}
Note that our first claim can also be written $\tilde{z}^{(n)}\simeq
\mathring{z}^{(n)}$ where $\tilde{z}^{(n)}$ is from (\ref{n-vec-of-ind}). To
describe $\mathcal{L}(\tilde{z}^{(n)}|f)$, define $\delta_{j}=\tilde{f}%
_{n}(\omega_{j-(n+1)/2})$ for $j=1,\ldots,n$ and a $n\times n$ covariance
matrix%
\[
\Delta_{n}(f)=Diag\left(  \delta_{j}\right)  _{j=1,\ldots,n}.
\]
Then $\mathcal{L}(\tilde{z}^{(n)}|f)=N_{n}(0,\Delta_{n}(f))$. The conditions
on $f$ (see also Lemma \ref{lem-sobol-embed}\ Appendix)\ imply that uniformly
over $j=1,\ldots,n$
\[
J_{j,n}(f)\geq C^{-1}\text{, }J_{j,n}(f)\leq C
\]
for some $C>0$ not depending on $f$ and $n$. Now apply Lemma
\ref{lem-helldist-covmatr} to obtain%
\[
H^{2}\left(  N_{n}(0,\mathring{\Gamma}_{n}(f)),N_{n}(0,\Delta_{n}(f))\right)
\leq C\;\left\Vert \mathring{\Gamma}_{n}(f)-\Delta_{n}(f)\right\Vert
^{2}=C\;\sum_{j=1}^{n}\left(  J_{j,n}(f)-\delta_{j}\right)  ^{2}.
\]
By\ Lemma \ref{lem-midpoints-appr} this is $o(1)$ uniformly in $f$. This
implies the first relation $\simeq$. The second relation is an obvious consequence.
\end{proof}

For a choice $m=n+r_{n}$, $r_{n}=2\left[  \log(n/2)\right]  $ we immediately
obtain the following result. Define the upper bracket Gaussian scale
experiment $\mathcal{\mathring{E}}_{u,n}$ by
\begin{equation}
\mathcal{\mathring{E}}_{u,n}:=\mathcal{\mathring{E}}_{n+r_{n}}%
.\label{upp-bracket-periodic exp}%
\end{equation}

\begin{corollary}
\label{cor-upper-brack}Consider experiments $\mathcal{E}_{n}$ and
$\mathcal{\mathring{E}}_{u,n}$ given respectively by (\ref{exper-1}) and
(\ref{upp-bracket-periodic exp}), (\ref{exp-e-dot-def}) with parameter space
$\Sigma=\Sigma_{\alpha,M}$ where $M>0$, $\alpha>1/2$. Then as $n\rightarrow
\infty$
\[
\mathcal{E}_{n}\precsim\mathcal{\mathring{E}}_{u,n}.
\]

\end{corollary}

\section{Lower informativity bracket}

The upper bound (\ref{upperbound-periodic-orig}) for the Hellinger distance of
$y^{(n)}$ and the periodic process $\tilde{y}^{(n)}$ which does not tend to
$0$, can be improved in a certain sense if $f$ is restricted to a shrinking
neighborhod, $\Sigma_{n}(f_{0})$ say, of some $f_{0}\in\Sigma$. At this stage,
$f_{0}$ is assumed known so the covariance matrices $\Gamma_{n}(f)$ and
$\tilde{\Gamma}_{n}(f)$ can be used for a linear transformation of $y^{(n)}$
which brings it closer to the periodic process $\tilde{y}^{(n)}$. The linear
transformation of $y^{(n)}$ which depends on $f_{0}$ can be construed as a
Markov kernel mapping which yields asymptotic equivalence $\mathcal{E}%
_{n}(f_{0})\approx$ $\mathcal{\tilde{E}}_{n}(f_{0})$ if these are the versions
of $\mathcal{E}_{n}$ and $\mathcal{\tilde{E}}_{n}$ with $f$ restricted to
$f\in\Sigma_{n}(f_{0})$.

Such a local asymptotic equivalence can be globalized in a standard way (cf.
Nussbaum (1996), Grama and Nussbaum (1998)) if \textit{sample splitting} were
available in both global experiments $\mathcal{E}_{n}$ and $\mathcal{\tilde
{E}}_{n}$. For the original stationary process that would mean that observing
a series of size $n$ is equivalent to observing two independent series of size
approximately $n/2$. We will establish an asymptotic version of sample
splitting for $y^{(n)}$ which involves omitting a fraction of the sample in
the center of the series, i.e. omitting terms with index near $n/2$. The
ensuing loss of information means that the globalization procedure only yields
a lower asymptotic informativity bracket for $\mathcal{E}_{n}$, i.e. a
sequence $\mathcal{\tilde{E}}_{3,n}^{\#}$ such that $\mathcal{\tilde{E}}%
_{3,n}^{\#}\precsim\mathcal{E}_{n}$. The experiment $\mathcal{\tilde{E}}%
_{3,n}^{\#}$ will be made up of two independent periodic processes with the
same parameter $f$ and with a sample size $m\sim(n-\log n)/2.$ Each of these
is equivalent to a Gaussian scale model (\ref{n-vec-of-ind}) with $n$ replaced
by $m$ ; further arguments show that observing these two is asymptotically
equivalent to a Gaussian scale model $\mathcal{\mathring{E}}_{l,n}%
:=\mathcal{\mathring{E}}_{2m}$ with grid size $2m\sim n-\log n$.

A crucial step now consists in showing that in the Gaussian scale models
$\mathcal{\mathring{E}}_{n}$, the grid size $n$ can be replaced by $n-\log n$
or $n+\log n$. This step is an analog, for the special regression model, of
the well known reasoning in the i.i.d. case that additional observations may
be asymptotically negligible (cf. Mammen (1986) for parametric i.i.d. models,
Low and Zhou (2004) for the nonparametric case). Thus it follows that the
lower and upper bracketing experiments $\mathcal{\mathring{E}}_{l,n} $,
$\mathcal{\mathring{E}}_{u,n}$ are both asymptotically equivalent to
$\mathcal{\mathring{E}}_{n}$, and the relations
\[
\mathcal{\mathring{E}}_{l,n}\precsim\mathcal{E}_{n}\precsim\mathcal{\mathring
{E}}_{u,n}%
\]
then imply $\mathcal{E}_{n}\approx\mathcal{\mathring{E}}_{n}$, i.e. Theorem
\ref{theor-main-1}.

\subsection{Local experiments}

Let $\varkappa_{n}$ be a sequence $\varkappa_{n}\searrow0$, fixed in the
sequel. A specific choice of $\varkappa_{n}$ will be made in section
\ref{sec-global} below (see (\ref{kappa-select})). Let $\left\Vert
\cdot\right\Vert _{\infty}$ be the sup-norm for real functions defined on
$[-\pi,\pi]$, i.e.
\[
\left\Vert f\right\Vert _{\infty}=\sup_{\omega\in\lbrack-\pi,\pi]}\left\vert
f(\omega)\right\vert
\]
and for $f_{0}\in\Sigma$ define shrinking neighborhoods
\begin{equation}
\Sigma_{n}(f_{0})=\left\{  f\in\Sigma:\left\Vert f-f_{0}\right\Vert _{\infty
}+\left\Vert f-f_{0}\right\Vert _{2,1/2}\leq\varkappa_{n}\right\}
.\label{shrink-neighb-def}%
\end{equation}
The restricted experiments are
\[
\mathcal{E}_{n}(f_{0})=\left(  N_{n}(0,\Gamma_{n}(f)),f\in\Sigma_{n}%
(f_{0})\right)  ,\;\mathcal{\tilde{E}}_{n}(f_{0})=\left(  N_{n}(0,\tilde
{\Gamma}_{n}(f)),f\in\Sigma_{n}(f_{0})\right)  .
\]
For shortness write $\Gamma=\Gamma_{n}(f)$, $\Gamma_{0}=\Gamma_{n}(f_{0}) $
and similarly $\tilde{\Gamma}=\Gamma_{n}(f)$, $\tilde{\Gamma}_{0}%
=\tilde{\Gamma}_{n}(f_{0})$. Define a matrix
\begin{equation}
K_{n}=K_{n}(f_{0})=\tilde{\Gamma}_{0}^{1/2}\Gamma_{0}^{-1/2}\label{def-K}%
\end{equation}
and in experiment $\mathcal{E}_{n}(f_{0})$ consider transformed observations
\[
\check{y}^{(n)}:=K_{n}(f_{0})y^{(n)}.
\]
Consider also the experiment $\mathcal{E}_{n}^{\ast}(f_{0})$ given by the laws
of $\check{y}^{(n)}$, i.e.
\[
\mathcal{E}_{n}^{\ast}(f_{0})=\left(  N_{n}(0,K_{n}(f_{0})\Gamma_{n}%
(f)K_{n}^{\prime}(f_{0})),f\in\Sigma_{n}(f_{0})\right)  .
\]
Clearly $\mathcal{E}_{n}(f_{0})\sim\mathcal{E}_{n}^{\ast}(f_{0})$; the next
result proves that $\mathcal{E}_{n}^{\ast}(f_{0})\simeq\mathcal{\tilde{E}}%
_{n}(f_{0})$ and thus $\mathcal{E}_{n}(f_{0})\approx\mathcal{\tilde{E}}%
_{n}(f_{0})$.

\begin{lemma}
\label{lem-local-likproc} We have
\[
\sup_{f_{0}\in\Sigma}\sup_{f\in\Sigma_{n}(f_{0})}H^{2}\left(  N_{n}%
(0,K_{n}(f_{0})\Gamma_{n}(f)K_{n}^{\prime}(f_{0}),N_{n}(0,\tilde{\Gamma}%
_{n}(f)\right)  \leq C\;\varkappa_{n}\text{.}%
\]

\end{lemma}

\begin{proof}
\textbf{\ }In view of Lemma \ref{lem-helldist-covmatr}, it suffices to show
that
\begin{equation}
\sup_{f\in\Sigma}\left(  \lambda_{\max}(\tilde{\Gamma}_{n})+\lambda_{\min
}^{-1}(\tilde{\Gamma}_{n})\right)  \leq C\label{prop-eigenvalues}%
\end{equation}
and that
\[
\left\Vert K_{n}\Gamma_{n}K_{n}^{\prime}-\tilde{\Gamma}_{n}\right\Vert
^{2}\leq C\;\varkappa_{n}.
\]
Note that
\[
\lambda_{\max}(\tilde{\Gamma})=\max_{|j|\leq(n-1)/2}\left\vert \tilde{f}%
_{n}(\omega_{j})\right\vert ,\;\;\;\;\lambda_{\min}(\tilde{\Gamma}%
)=\min_{|j|\leq(n-1)/2}\left\vert \tilde{f}_{n}(\omega_{j})\right\vert
\]
and that Lemma \ref{lem-sobol-embed} implies
\[
\sup{}_{f\in\Sigma}\left\Vert f-\tilde{f}_{n}\right\Vert _{\infty}%
\rightarrow0.
\]
Hence (\ref{prop-eigenvalues})\ follows immediately from $f\in\Sigma$, more
specifically the fact that values of $f$ are uniformly bounded and bounded
away from $0$. According to Proposition 4.5.3 in Brockwell and Davis (1991),
the assumption $f\in\Sigma$ also implies a corresponding property for $\Gamma
$, i.e.
\begin{equation}
\sup_{f\in\Sigma}\left(  \lambda_{\max}(\Gamma_{n})+\lambda_{\min}^{-1}%
(\Gamma_{n})\right)  \leq C.\label{condit-eigenval-gamma}%
\end{equation}
Note that eigenvalues of $\Gamma_{0}$ and $\tilde{\Gamma}_{0}$ share property
(\ref{prop-eigenvalues}) since $f_{0}\in\Sigma$.

Set $G=\Gamma_{0}^{-1/2}\Gamma\Gamma_{0}^{-1/2}$ and $\tilde{G}=\tilde{\Gamma
}_{0}^{-1/2}\tilde{\Gamma}\tilde{\Gamma}_{0}^{-1/2}$. Since
\[
\left\Vert K_{n}\Gamma_{n}K_{n}^{\prime}-\tilde{\Gamma}_{n}\right\Vert
\leq\left\vert \tilde{\Gamma}_{0}\right\vert \left\Vert G-\tilde
{G}\right\Vert
\]
it now suffices to show that
\begin{equation}
\left\Vert G-\tilde{G}\right\Vert \leq C\;\varkappa_{n}.\label{zeroconverg-S}%
\end{equation}
To establish (\ref{zeroconverg-S}), denote $\Delta=\Gamma-\Gamma_{0}$,
$\ \tilde{\Delta}=\tilde{\Gamma}-\tilde{\Gamma}_{0}$ and observe \textbf{\ }
\begin{align}
\left\Vert G-\tilde{G}\right\Vert  & =\left\Vert \Gamma_{0}^{-1/2}\Gamma
\Gamma_{0}^{-1/2}-\tilde{\Gamma}_{0}^{-1/2}\tilde{\Gamma}\tilde{\Gamma}%
_{0}^{-1/2}\right\Vert \nonumber\\
& =\left\Vert \Gamma_{0}^{-1/2}\Delta\Gamma_{0}^{-1/2}-\tilde{\Gamma}%
_{0}^{-1/2}\tilde{\Delta}\tilde{\Gamma}_{0}^{-1/2}\right\Vert \nonumber\\
& \leq\left\Vert \Gamma_{0}^{-1/2}\left(  \Delta-\tilde{\Delta}\right)
\Gamma_{0}^{-1/2}\right\Vert +\left\Vert \Gamma_{0}^{-1/2}\tilde{\Delta}%
\Gamma_{0}^{-1/2}-\tilde{\Gamma}_{0}^{-1/2}\tilde{\Delta}\tilde{\Gamma}%
_{0}^{-1/2}\right\Vert .\label{twotermsrightside}%
\end{align}
We shall now estimate the two terms on the right side separately. By
elementary properties of eigenvalues we obtain
\[
\left\Vert \Gamma_{0}^{-1/2}\left(  \Delta-\tilde{\Delta}\right)  \Gamma
_{0}^{-1/2}\right\Vert \leq\left\vert \Gamma_{0}^{-1}\right\vert \left\Vert
\Delta-\tilde{\Delta}\right\Vert
\]
where $\left\vert \Gamma_{0}^{-1}\right\vert \leq C$ and according to Lemma
\ref{lem-covmatnorm-sobol} (ii)
\[
\left\Vert \Delta-\tilde{\Delta}\right\Vert ^{2}\leq2\left\vert f-f_{0}%
\right\vert _{2,1/2}^{2}.
\]
Furthermore
\begin{align*}
& \left\Vert \Gamma_{0}^{-1/2}\tilde{\Delta}\Gamma_{0}^{-1/2}-\tilde{\Gamma
}_{0}^{-1/2}\tilde{\Delta}\tilde{\Gamma}_{0}^{-1/2}\right\Vert \\
& =\left\Vert \left(  \Gamma_{0}^{-1/2}-\tilde{\Gamma}_{0}^{-1/2}\right)
\tilde{\Delta}\Gamma_{0}^{-1/2}+\tilde{\Gamma}_{0}^{-1/2}\tilde{\Delta}\left(
\Gamma_{0}^{-1/2}-\tilde{\Gamma}_{0}^{-1/2}\right)  \right\Vert \\
& \leq2C\left\vert \tilde{\Delta}\right\vert \left\Vert \Gamma_{0}%
^{-1/2}-\tilde{\Gamma}_{0}^{-1/2}\right\Vert =C\left\vert \tilde{\Delta
}\right\vert \left\Vert \Gamma_{0}^{-1/2}\left(  \tilde{\Gamma}_{0}%
^{1/2}-\Gamma_{0}^{1/2}\right)  \tilde{\Gamma}_{0}^{-1/2}\right\Vert \\
& \leq C\left\vert \tilde{\Delta}\right\vert \left\Vert \Gamma_{0}%
^{1/2}-\tilde{\Gamma}_{0}^{1/2}\right\Vert .
\end{align*}
Applying Lemma \ref{lem-sqrt-matr} and Lemma \ref{lem-covmatnorm-sobol} (i) we
obtain
\[
\left\Vert \Gamma_{0}^{1/2}-\tilde{\Gamma}_{0}^{1/2}\right\Vert ^{2}\leq
C\left\Vert \Gamma_{0}-\tilde{\Gamma}_{0}\right\Vert ^{2}\leq C\left\vert
f_{0}\right\vert _{2,1/2}^{2}.
\]
Here $\left\vert f_{0}\right\vert _{2,1/2}^{2}\leq\left\vert f_{0}\right\vert
_{2,\alpha}^{2}\leq M$. Collecting these estimates yields
\[
\left\Vert G-\tilde{G}\right\Vert ^{2}\leq C\left(  \left\vert f-f_{0}%
\right\vert _{2,1/2}^{2}+\left\vert \tilde{\Delta}\right\vert ^{2}\right)  .
\]
To complete the proof, it suffices to note that, since $\tilde{\Gamma}$ and
$\tilde{\Gamma}_{0}$ have the same set of eigenvectors (cf.
(\ref{spec-decomp-gamma-tilde}) and (\ref{ev-1})-(\ref{ev-3}))
\begin{align*}
\left\vert \tilde{\Delta}\right\vert ^{2}  & =\lambda_{\max}(\tilde{\Gamma
}-\tilde{\Gamma}_{0})^{2}=(2\pi)^{2}\max_{|j|\leq(n-1)/2}\left(  \left\vert
\tilde{f}_{n}(\omega_{j})-\tilde{f}_{0,n}(\omega_{j})\right\vert ^{2}\right)
\\
& \leq C\left\Vert \tilde{f}_{n}-\tilde{f}_{0,n}\right\Vert _{\infty}^{2}\leq
C\left\Vert f-f_{0}\right\Vert _{\infty}^{2}+C\;n^{1-2\alpha}\log
n\;\left\Vert f-f_{0}\right\Vert _{2,\alpha}^{2}%
\end{align*}
where the last inequality is a consequence of Lemma \ref{lem-sobol-embed}.
Hence $\left\vert \tilde{\Delta}\right\vert \leq C\varkappa_{n}$, which
establishes (\ref{zeroconverg-S}).
\end{proof}

\subsection{Sample splitting}

Consider sample splitting for a stationary process: Take the observed
$y^{(n)}=(y(1),\ldots,y(n))$ and omit $r$ observations in the center of the
series. Recall that $n$ was assumed uneven; assume now also $r$ to be uneven
and set $m=(n-r)/2$, then the result is the series $y(1),\ldots,y(m),$
$y(n-m+1),\ldots,y(n)$. The total covariance matrix for these reduced data is
\[
\Gamma_{n,0}^{(m)}(f):=\left(
\begin{array}
[c]{cc}%
\Gamma_{m}(f) & {\normalsize A}_{n,m}\\
{\normalsize A}_{n,m}^{\prime} & \Gamma_{m}(f)
\end{array}
\right)
\]
where the $m\times m$ matrix $A_{n,m}=A_{n,m}(f)$ contains only covariances
$\gamma_{f}(r+1),$ $\gamma_{f}(r+2)$ and of higher order. In fact $A$ is the
upper right $m\times m$ submatrix of $\Gamma_{n}(f)$, i.e.
\[
A_{n,m}\mathbf{=}\left(
\begin{array}
[c]{ccc}%
\ldots & \gamma(n-2) & \gamma(n-1)\\
\gamma(r+2) & \ldots & \gamma(n-2)\\
\gamma(r+1) & \gamma(r+2) & \ldots
\end{array}
\right)  .
\]
In the sequel we set $r_{n}=2[\log n/2]+1$ and thus $r_{n}\sim\log n$,
$m=\left(  n-r_{n}\right)  /2$. The corresponding experiment we denote
\[
\mathcal{E}_{0,n}^{\#}=\left(  N_{2m}(0,\Gamma_{n,0}^{(m)}(f)),f\in
\Sigma\right)
\]
Consider also the experiment where two independent stationary series of length
$m$ are observed, $y_{1}^{(m)}$ and $y_{2}^{(m)}$, say. The corresponding
experiment is%
\begin{equation}
\mathcal{E}_{1,n}^{\#}:=\left(  N_{2m}(0,\Gamma_{n,1}^{(m)}(f)),f\in
\Sigma\right) \label{two-ind}%
\end{equation}
where
\[
\Gamma_{n,1}^{(m)}(f):=\left(
\begin{array}
[c]{cc}%
\Gamma_{m}(f) & 0_{m\times m}\\
0_{m\times m} & \Gamma_{m}(f)
\end{array}
\right)  .
\]
Clearly we have $\mathcal{E}_{0,n}^{\#}\preceq\mathcal{E}_{n}$.

\begin{proposition}
$\mathcal{E}_{0,n}^{\#}\simeq\mathcal{E}_{1,n}^{\#}.$
\end{proposition}

\begin{proof}
Use Lemma \ref{lem-helldist-covmatr} to compute the Hellinger distance. Take
$A=\Gamma_{n,1}^{(m)}$; then the eigenvalues of $A$ are those of $\Gamma
_{m}(f)$, so that (\ref{condit-eigenval-gamma}) can be invoked. The squared
distance of the covariance matrices $\Gamma_{n,0}^{(m)}$ and $\Gamma
_{n,1}^{(m)}$ is
\[
\left\Vert \Gamma_{n,0}^{(m)}-\Gamma_{n,1}^{(m)}\right\Vert ^{2}=2\left\Vert
A_{n,m}\right\Vert ^{2}\leq2\sum_{k=r+1}^{n-1}(k-r)\gamma^{2}(k)
\]%
\[
\leq2\sum_{k=r+1}^{n-1}k\gamma^{2}(k)\leq(r+1)^{1-2\alpha}\left\vert
f\right\vert _{2,\alpha}^{2}.
\]
Since $r_{n}\rightarrow\infty$, the result follows.
\end{proof}

We have shown that two independent stationary sequences of length $m=\left(
n-r_{n}\right)  /2$ are asymptotically less informative than one sequence of
length $n$. Having obtained a method of sample splitting for stationary
sequences (with some loss of information), we can now use a localization
argument to complete the proof of the lower bound.

\subsection{Preliminary estimators}

For the globalization procedure, we need existence of an estimator $\hat
{f}_{n}$, in both of the global experiments $\mathcal{E}_{n}$ and
$\mathcal{\tilde{E}}_{n}$ (or $\mathcal{\mathring{E}}_{n}$), such that
$\hat{f}_{n}$ takes values in $\Sigma$ and
\[
\left\Vert \hat{f}_{n}-f\right\Vert _{\infty}+\left\Vert \hat{f}%
_{n}-f\right\Vert _{2,1/2}=o_{p}(1)
\]
uniformly over $f\in\Sigma$. More specifically, a rate $o_{p}(\kappa_{n})$
with $\kappa_{n}$ from (\ref{shrink-neighb-def})\ is needed in the above
result, but $\kappa_{n}$ has not been selected so far, and will be determined
based on the results of this section (cf. (\ref{kappa-select}) below). Select
$\beta\in(1/2,\alpha)$ and consider the norm $\left\Vert f\right\Vert
_{2,\beta}$ according to (\ref{normdef}). Note that $\left\Vert f\right\Vert
_{2,1/2}\leq C\left\Vert f\right\Vert _{2,\beta}$ and that according to Lemma
\ref{lem-sobol-embed}\textbf{\ } $\left\Vert f\right\Vert _{\infty}\leq
C\left\Vert f\right\Vert _{2,\beta}$; therefore it suffices to show
\begin{equation}
\left\Vert \hat{f}_{n}-f\right\Vert _{2,\beta}^{2}=o_{p}%
(1).\label{consist-beta}%
\end{equation}
For this, we shall use a standard truncated orthogonal series estimator and
then modify it to take values in $\Sigma$. The empirical autocovariance
function is
\[
\hat{\gamma}_{n}(k)=\frac{1}{n-k}\sum_{j=1}^{n-k}y(j)y(k+j)\text{, }%
k=0,\ldots,n-1.
\]
We have unbiasedness: $E\hat{\gamma}_{n}(k)=\gamma_{f}(k)$; for the variance
of $\hat{\gamma}_{n}(k)$ we have the following result.

\begin{lemma}
\label{lem-variance-bound-estim}For any spectral density $f\in L_{2}(-\pi
,\pi)$, and any $k=0,\ldots,n-1$
\[
Var\hat{\gamma}_{n}(k)\leq\frac{5}{n-k}\sum_{j=0}^{n-1}\gamma_{f}^{2}(j).
\]

\end{lemma}

\begin{proof}
For given $k$, set $m=n-k$ and $z(j)=y(j)y(j+k)-\gamma_{f}(k),$ $j=1,\ldots
,m$. The $z(j)$ form a zero mean stationary series, with autocovariance
function $\rho(j)$, say. We have
\begin{align}
E\left(  \frac{1}{m}\sum_{k=1}^{m}z(k)\right)  ^{2}  & =\frac{1}{m^{2}}%
\sum_{1\leq j,k\leq m}\rho(k-j)=\frac{1}{m}\rho(0)+2\frac{1}{m^{2}}\sum
_{k=1}^{m-1}(m-k)\rho(k)\nonumber\\
& \leq\frac{2}{m}\sum_{k=0}^{m-1}\rho(k).\label{covar-estim-1}%
\end{align}
The computation in Shiryaev (1996), (VI.4.5-6) gives
\[
\rho(j)=\gamma^{2}(j)+\gamma(j-k)\gamma(j+k).
\]
The inequality
\[
2\left\vert \gamma(j-k)\gamma(j+k)\right\vert \leq\gamma^{2}(j-k)+\gamma
^{2}(j+k)
\]
now implies
\[
\sum_{k=0}^{m-1}\rho(k)\leq\frac{5}{2}\sum_{k=0}^{n-1}\gamma^{2}(k)
\]
(we bound the sum involving $\gamma^{2}(j-k)$ by $2\sum_{j=0}^{n-1}\gamma
^{2}(j)$). In conjunction with (\ref{covar-estim-1}) this proves the lemma.
\end{proof}

For the orthogonal series estimator, define a truncation index $\tilde
{n}=[n^{1/(2\alpha+1)}]$ and set
\begin{equation}
\hat{f}_{n}(\omega)=\sum_{|k|\leq\tilde{n}}\hat{\gamma}_{n}(k)\mathrm{\exp
}{\normalsize (\mathrm{i}k\omega),\;\omega\in\lbrack-\pi,\pi].}%
\label{trunc-os-estimator}%
\end{equation}

\begin{lemma}
\label{lem-prelim-est-basic}In the experiment $\mathcal{E}_{n}$ the estimator
$\hat{f}_{n}$ fulfills for any $\beta\in(1/2,\alpha)$ and any $\gamma
\in\left(  0,\frac{\alpha-\beta}{2\alpha+1}\right)  $
\begin{equation}
\sup_{f\in\Sigma}P\left(  \left\Vert \hat{f}_{n}-f\right\Vert _{2,\beta}%
^{2}>n^{-\gamma}\right)  \rightarrow0.\label{rate-converg-prelim-est}%
\end{equation}

\end{lemma}

\begin{proof}
By the Markov inequality, it suffices to prove
\[
\sup_{f\in\Sigma}E_{f}\left\Vert \hat{f}_{n}-f\right\Vert _{2,\beta}%
^{2}=o(n^{-\gamma}).
\]
A bias-variance decomposition and Lemma \ref{lem-variance-bound-estim} yield%
\begin{align*}
E_{f}\left\Vert \hat{f}_{n}-f\right\Vert _{2,\beta}^{2}  & =\sum
_{|k|\leq\tilde{n}}\max\left(  1,|k|^{2\beta}\right)  Var\hat{\gamma}%
_{n}(k)+\sum_{|k|>\tilde{n}}|k|^{2\beta}\gamma^{2}(k)\\
& \leq\sum_{|k|\leq\tilde{n}}\max\left(  1,|k|^{2\beta}\right)  \frac{5}%
{n-k}\left(  \sum_{j=0}^{n-1}\gamma_{f}^{2}(j)\right)  +\tilde{n}%
^{2\beta-2\alpha}\sum_{|k|>\tilde{n}}|k|^{2\alpha}\gamma^{2}(k)\\
& \leq\frac{C}{n}\left\Vert f\right\Vert _{2}^{2}\sum_{|k|\leq\tilde{n}}%
\max\left(  1,|k|^{2\beta}\right)  +\tilde{n}^{2\beta-2\alpha}\left\vert
f\right\vert _{2,\alpha}^{2}\\
& \leq C\left\Vert f\right\Vert _{2}^{2}n^{-1}\tilde{n}^{2\beta+1}+C\left\vert
f\right\vert _{2,\alpha}^{2}\tilde{n}^{2(\beta-\alpha)}\\
& \leq C\left(  \left\Vert f\right\Vert _{2}^{2}+\left\vert f\right\vert
_{2,\alpha}^{2}\right)  n^{2(\beta-\alpha)/(2\alpha+1)}.
\end{align*}
Since $\left\Vert f\right\Vert _{2}^{2}\leq C\left\Vert f\right\Vert
_{2,\alpha}^{2}$ and $\left\vert f\right\vert _{2,\alpha}^{2}\leq\left\Vert
f\right\Vert _{2,\alpha}^{2}$, the result follows.
\end{proof}

We now turn to preliminary estimation in the periodic experiment
$\mathcal{\tilde{E}}_{n}$ with data vector $\tilde{y}^{(n)}$. Note that this
data vector can be construed as coming from a stationary sequence with
autocoviance function $\tilde{\gamma}_{(n)}(\cdot)$ given by
(\ref{gamma-n-tilde-def}) for $|k|\leq n-1$ and $\tilde{\gamma}_{(n)}(k)=0$
for $|k|>n-1$, i.e. the stationary sequence having spectral density $\tilde
{f}_{n} $. Thus if $\hat{\gamma}_{n}(k)$ again denotes the empirical
autocoviance function in this series then we can apply Lemma
\ref{lem-variance-bound-estim} to obtain
\[
Var\hat{\gamma}_{n}(k)\leq\frac{5}{n-k}\sum_{j=0}^{n-1}\tilde{\gamma}%
_{(n),f}^{2}(j),\;k=0,\ldots,n-1.
\]
Obviously
\[
\sum_{k=0}^{n-1}\tilde{\gamma}_{(n),f}^{2}(k)=\sum_{k=0}^{(n-1)/2}\gamma
_{f}^{2}(k)+\sum_{k=1}^{(n-1)/2}\gamma_{f}^{2}(k)\leq2\left\Vert f\right\Vert
_{2}^{2}.
\]
Now use the estimator (\ref{trunc-os-estimator}) with $\tilde{n}$ as above;
since $\tilde{n}=o((n-1)/2)$, we have the unbiasedness%
\[
E\hat{\gamma}_{n}(k)=\gamma_{f}(k)\text{, }k=0,\ldots,\tilde{n}.
\]
Thus the proof of the following result is entirely analogous to Lemma
\ref{lem-prelim-est-basic}; the estimator $\hat{f}_{n}$ is also formally the
same function of the data.

\begin{lemma}
\label{lem-prelim-est-basic-2}In the experiment $\mathcal{\tilde{E}}_{n}$ the
estimator $\hat{f}_{n}$ fulfills (\ref{rate-converg-prelim-est}) for any
$\beta\in(1/2,\alpha)$ and any $\gamma\in\left(  0,\frac{\alpha-\beta}%
{2\alpha+1}\right)  $.
\end{lemma}

Finally consider modifications such that the estimator takes values in
$\Sigma_{\alpha,M}$. Consider the space $W^{\beta}=\left\{  f\in L_{2}%
(-\pi,\pi):\left\Vert f\right\Vert _{2,\beta}^{2}<\infty\right\}  $; this is a
periodic fractional Sobolev space which is Hilbert under the norm $\left\Vert
f\right\Vert _{2,\beta}$. There the set $\Sigma_{\alpha,M}$ is compact and
convex; hence there exists a ($\left\Vert \cdot\right\Vert _{2,\beta}%
$-continuous) projection operator $\Pi\ $onto $\Sigma_{\alpha,M}$ in
$W^{\beta}$ (cf. Balakrishnan (1976), Definition 1.4.1 ). Then
\[
\left\Vert \Pi\left(  \hat{f}_{n}\right)  -f\right\Vert _{2,\beta}%
\leq\left\Vert \hat{f}_{n}-f\right\Vert _{2,\beta}.
\]
The modified estimators $\Pi\left(  \hat{f}_{n}\right)  $ thus again fulfill
(\ref{rate-converg-prelim-est}). A summary of results in this section is the following.

\begin{proposition}
\label{prop-prelim-est}In both experiments $\mathcal{E}_{n}$ and
$\mathcal{\tilde{E}}_{n}$ there are estimators $\hat{f}_{n}$ taking values in
$\Sigma$ and fulfilling for any $\gamma\in\left(  0,\frac{\alpha-1/2}%
{2\alpha+1}\right)  $
\[
\sup_{f\in\Sigma}P\left(  \left\Vert \hat{f}_{n}-f\right\Vert _{\infty
}+\left\Vert \hat{f}_{n}-f\right\Vert _{2,1/2}>n^{-\gamma}\right)
\rightarrow0.
\]

\end{proposition}

\subsection{Globalization\label{sec-global}}

In this section we denote
\[
P_{f,n}:=\mathcal{L}(y^{(n)}|f)=N_{n}\left(  0,\Gamma_{n}(f)\right)
,\;\;\tilde{P}_{f,n}:=\mathcal{L}(\tilde{y}^{(n)}|f)=N_{n}\left(
0,\tilde{\Gamma}_{n}(f)\right)
\]
Consider again the experiment $\mathcal{E}_{1,n}^{\#}$ of (\ref{two-ind} where
two independent stationary series $y_{1}^{(m)}$ and $y_{2}^{(m)}$ of length
$m=\left(  n-r_{n}\right)  /2$ are observed. In modified notation we now
write
\[
\mathcal{E}_{1,n}^{\#}=\mathcal{E}_{m}\otimes\mathcal{E}_{m}=\left(
P_{f,m}\otimes P_{f,m},\;f\in\Sigma\right)  .
\]
We shall compare this with the experiments%
\begin{align*}
\mathcal{E}_{2,n}^{\#}  & :=\mathcal{E}_{m}\otimes\mathcal{\tilde{E}}%
_{m}=\left(  P_{f,m}\otimes\tilde{P}_{f,m},\;f\in\Sigma\right)  ,\\
\mathcal{E}_{3,n}^{\#}  & :=\mathcal{\tilde{E}}_{m}\otimes\mathcal{\tilde{E}%
}_{m}=\left(  \tilde{P}_{f,m}\otimes\tilde{P}_{f,m},\;f\in\Sigma\right)  .
\end{align*}
At this point select the shrinking rate $\kappa_{n}$ of the neighborhoods
$\Sigma_{n}(f_{0})$ (cp. (\ref{shrink-neighb-def})) as
\begin{equation}
\kappa_{n}=n^{-\gamma}\text{, }\gamma=\frac{\alpha-1/2}{2(2\alpha
+1)}\label{kappa-select}%
\end{equation}

\begin{proposition}
We have $\mathcal{E}_{2,n}^{\#}\precsim\mathcal{E}_{1,n}^{\#}.$
\end{proposition}

\begin{proof}
We shall construct a sequence of Markov kernels $M_{n}$ such that
\[
\sup_{f\in\Sigma}H^{2}\left(  P_{f,m}\otimes\tilde{P}_{f,m},M_{n}\left(
P_{f,m}\otimes P_{f,m}\right)  \right)  \rightarrow0.
\]
Define $M_{n}$ as follows: given $y_{1}^{(m)}$ and $y_{2}^{(m)}$, and $A$, a
measurable subset of $\mathbb{R}^{2m}$, set
\[
M_{n}\left(  A,y_{1}^{(m)},y_{2}^{(m)}\right)  =\mathbf{1}_{A}\left(
y_{1}^{(m)},K_{m}(\hat{f}_{m}(y_{1}^{(m)}))y_{2}^{(m)}\right)
\]
where $K_{m}(f)$ is the matrix defined by (\ref{def-K}), i.e. for $f\in\Sigma$
by
\[
K_{m}(f)=\tilde{\Gamma}_{m}^{1/2}(f)\Gamma_{m}^{-1/2}(f)
\]
and $\hat{f}_{m}$ is the estimator in $\mathcal{E}_{m}$ of Proposition
\ref{prop-prelim-est} applied to data $y_{1}^{(m)}$. Thus the Markov kernel
$M_{n}$ is in fact a deterministic map, i.e. given $y_{1}^{(m)},y_{2}^{(m)}$,
it defines a one point measure on $\mathbb{R}^{2m}$ concentrated in
$\allowbreak\left(  y_{1}^{(m)},K_{m}(\hat{f}_{m}(y_{1}^{(m)}))y_{2}%
^{(m)}\right)  $. Thus the law $M_{n}\left(  P_{f,m}\otimes P_{f,m}\right)  $
is the joint law of $y_{1}^{(m)}$ and $K_{m}(\hat{f}_{m}(y_{1}^{(m)}%
))y_{2}^{(m)}$ under $f$. The latter we split up into the marginal law of
$y_{1}^{(m)}$, i.e. $P_{f,m}$ and the conditional law of $K_{m}(\hat{f}%
_{m}(y_{1}^{(m)}))y_{2}^{(m)}$ given $y_{1}^{(m)}$; write $P_{f,m}^{K}%
|y_{1}^{(m)}$ for the latter. We have
\[
P_{f,m}^{K}|y_{1}^{(m)}=N_{n}\left(  0,K\Gamma_{m}(f)K^{\prime}\right)  \text{
for }K=K_{m}(\hat{f}_{m}(y_{1}^{(m)})).
\]
Now clearly
\begin{equation}
H^{2}\left(  P_{f,m}\otimes\tilde{P}_{f,m},M_{n}\left(  P_{f,m}\otimes
P_{f,m}\right)  \right)  =E_{f}H^{2}\left(  \tilde{P}_{f,m},P_{f,m}^{K}%
|y_{1}^{(m)}\right) \label{helldist-global-1}%
\end{equation}
where $E_{f}$ is taken wrt $y_{1}^{(m)}$ under $P_{f,m}$. Define
\[
B_{f,m}:=\left\{  y\in\mathbb{R}^{m}:\left\Vert \hat{f}_{m}(y)-f\right\Vert
_{\infty}+\left\Vert \hat{f}_{m}(y)-f\right\Vert _{2,1/2}\leq\kappa
_{m}\right\}  .
\]
By definition of $\Sigma_{m}(f_{0})$ (cf. (\ref{shrink-neighb-def})) we have
$f\in\Sigma_{m}(\hat{f}_{m}(y))$ if $y\in B_{f,m}$. Thus Lemma
\ref{lem-local-likproc} implies
\[
\sup_{y\in B_{f,m},f\in\Sigma}H^{2}\left(  \tilde{P}_{f,m},P_{f,m}%
^{K}|y\right)  =o(1).
\]
Moreover by Proposition \ref{prop-prelim-est}%
\begin{equation}
P_{f,m}\left(  B_{f,m}^{c}\right)  =o(1)\text{ uniformly over }f\in
\Sigma.\label{consist-1}%
\end{equation}
Hence
\begin{align}
E_{f}H^{2}\left(  \tilde{P}_{f,m},P_{f,m}^{K}|y_{1}^{(m)}\right)   &
=\int_{B_{f,m}}H^{2}\left(  \tilde{P}_{f,m},P_{f,m}^{K}|y\right)
dP_{f,m}(y)+o(1)\nonumber\\
& =o(1)P_{f,m}(B_{f,m})+o(1)=o(1)\label{reasoning-2}%
\end{align}
uniformly over $f\in\Sigma$. In conjunction with (\ref{helldist-global-1}) the
last relation proves the claim.
\end{proof}

The next result is entirely analogous if we replace the estimator $\hat{f}%
_{m}$ based on data $y^{(m)}$ by the one based on data $\tilde{y}^{(m)}$ and
formally reverse the order in the product $P_{f,m}\otimes\tilde{P}_{f,m}$.

\begin{proposition}
We have $\mathcal{E}_{3,n}^{\#}\precsim\mathcal{E}_{2,n}^{\#}.$
\end{proposition}

\begin{proof}
We construct a sequence of Markov kernels $\tilde{M}_{n}$ such that
\[
\sup_{f\in\Sigma}H^{2}\left(  \tilde{P}_{f,m}\otimes\tilde{P}_{f,m},\tilde
{M}_{n}\left(  P_{f,m}\otimes\tilde{P}_{f,m}\right)  \right)  \rightarrow0.
\]
Define $\tilde{M}_{n}$ as follows: given $y_{1}^{(m)}$ and $\tilde{y}%
_{2}^{(m)}$, and $A$, a measurable subset of $\mathbb{R}^{2m}$, set
\[
\tilde{M}_{n}\left(  A,y_{1}^{(m)},\tilde{y}_{2}^{(m)}\right)  =\mathbf{1}%
_{A}\left(  K_{m}(\hat{f}_{m}(\tilde{y}_{2}^{(m)}))y_{1}^{(m)},\tilde{y}%
_{2}^{(m)}\right)
\]
where $\hat{f}_{m}$ is the estimator defined in the previous subsection,
applied to data $\tilde{y}_{2}^{(m)}$. Analogously to (\ref{consist-1}) we
have
\[
\tilde{P}_{f,m}\left(  B_{f,m}^{c}\right)  =o(1)\text{ uniformly over }%
f\in\Sigma.
\]
A reasoning as in (\ref{reasoning-2}) completes the proof.
\end{proof}

For the experiment $\mathcal{E}_{3,n}^{\#}$ which consists of product measures
$\tilde{P}_{f,m}\otimes\tilde{P}_{f,m}$, we can invoke Proposition
\ref{prop-explicit-equiv-map}, applying the equivalence map given there
componentwise (i.e. to independent components $\left(  \tilde{y}_{1}%
^{(m)},\tilde{y}_{2}^{(m)}\right)  $ in $\mathcal{E}_{3,n}^{\#}$). A summary
of the lower informativity bound results so far can thus be given as follows.
For $r_{n}=2\left[  \log(n/2)\right]  $ define the lower bracket Gaussian
scale experiment $\mathcal{\mathring{E}}_{l,n}$ by
\begin{equation}
\mathcal{\mathring{E}}_{l,n}:=\mathcal{\mathring{E}}_{(n-r_{n})/2}%
\otimes\mathcal{\mathring{E}}_{(n-r_{n})/2}.\label{low-brack-scale-exp}%
\end{equation}

\begin{corollary}
\label{cor-lower-brack}Consider experiments $\mathcal{E}_{n}$ and
$\mathcal{\mathring{E}}_{l,n}$ given respectively by (\ref{exper-1}) and
(\ref{low-brack-scale-exp}), (\ref{exp-e-dot-def}) with parameter space
$\Sigma=\Sigma_{\alpha,M}$ where $M>0$, $\alpha>1/2$. Then as $n\rightarrow
\infty$
\[
\mathcal{\mathring{E}}_{l,n}\precsim\mathcal{E}_{n}.
\]

\end{corollary}

\subsection{Bracketing the Gaussian scale model}

The proof of Theorem \ref{theor-main-1} is complete if the lower and upper
informativity bounds $\mathcal{\mathring{E}}_{l,n}$ and $\mathcal{\mathring
{E}}_{u,n}$ coincide in an asymptotic sense. Since we already established the
relation $\mathcal{\mathring{E}}_{l,n}\precsim\mathcal{E}_{n}\precsim
\mathcal{\mathring{E}}_{u,n}$ (Corollaries \ref{cor-upper-brack},
\ref{cor-lower-brack}), it now suffices to show that $\mathcal{\mathring{E}%
}_{u,n}\precsim\mathcal{\mathring{E}}_{l,n}$. This essentially means that in
the special nonparametric regression model $\mathcal{\mathring{E}}_{n}$ of
Gaussian scale type, having $r_{n}$ additional observations does not matter
asymptotically. "Additional observations" here refers to an equidistant design
of higher grid size. The problem of additional observations for
i.~i.~d.~%
models has been discussed by Le Cam (1974) and Mammen (1986) under parametric
assumptions. For nonparametric
i.~i.~d.~%
models, one can use the approximation by Gaussian white noise or Poisson
models to bound the influence of additional observations. For simplicity,
consider a Gaussian white noise model on $[0,1]$
\[
dZ_{t}=f(t)dt+n^{-1/2}dW_{t}\text{, }t\in\lbrack0,1]\text{, }f\in\Sigma
\]
with parameter space $\Sigma$. Consider this experiment $\mathcal{F}_{n}$, say
and also $\mathcal{F}_{n+r_{n}}$. Multiplying the data by $n^{1/2}$ gives an
equivalent experiment
\[
dZ_{t}^{\ast}=n^{1/2}f(t)dt+dW_{t}\text{, }t\in\lbrack0,1]\text{, }f\in\Sigma
\]
and the corresponding one for $\left(  n+r_{n}\right)  ^{1/2}$. Now, for given
$f$, the squared Hellinger distance of the two respective measures is bounded
by
\begin{align*}
& C\left(  \left(  n+r_{n}\right)  ^{1/2}-n^{1/2}\right)  ^{2}\left\Vert
f\right\Vert _{2}^{2}\\
& =C\frac{r_{n}^{2}}{n}(1+o(1))\left\Vert f\right\Vert _{2}%
\end{align*}
if $r_{n}=o(n)$. Thus if $r_{n}=o(n^{1/2})$ and $\sup_{f\in\Sigma}\left\Vert
f\right\Vert _{2}\leq C$ then we have $\mathcal{F}_{n}\approx\mathcal{F}%
_{n+r_{n}}$.

Comparable results can be obtained for nonparametric
i.~i.~d.~%
and regression models if these can be approximated by $\mathcal{F}_{n}$. In
the present case, conversely, for the nonparametric Gaussian scale regression
$\mathcal{\mathring{E}}_{n}$, a result of type $\mathcal{\mathring{E}}%
_{n}\approx\mathcal{\mathring{E}}_{n+r_{n}}\mathcal{\ }$ is a prerequisite for
the Gaussian location (white noise) approximation. Note that for a narrower
parameter space, given by a Lipschitz class, the white noise approximation of
$\mathcal{\mathring{E}}_{n}$ has been established (cf. Grama and Nussbaum, 1998).

\begin{remark}
\emph{\label{rem-even}The relation }%
\begin{equation}
\mathcal{\mathring{E}}_{l,n}\precsim\mathcal{E}_{n}\precsim\mathcal{\mathring
{E}}_{u,n}\label{primary-bracket}%
\end{equation}
\emph{has been proved under the technical assumption that }$n$\emph{\ is
uneven. If }$n$\emph{\ is even, note first that }$\mathcal{E}_{n-1}%
\precsim\mathcal{E}_{n}\precsim\mathcal{E}_{n+1}$\emph{\ (omitting one
observation from }$\mathcal{E}_{n+1}$\emph{\ and }$\mathcal{E}_{n}$\emph{) and
apply (\ref{primary-bracket}) to obtain }%
\[
\mathcal{\mathring{E}}_{l,n-1}\precsim\mathcal{E}_{n}\precsim
\mathcal{\mathring{E}}_{u,n+1}%
\]
\emph{The relation }$\mathcal{E}_{u,n}\precsim\mathcal{E}_{l,n}$\emph{\ which
will be proved for uneven }$n$\emph{\ in the remainder of this section is
easily seen to extend to }$\mathcal{E}_{u,n+2}\precsim\mathcal{E}_{l,n}%
$\emph{. This suffices to establish the main result Theorem \ref{theor-main-1}
for general sample size }$n\rightarrow\infty$\emph{.}
\end{remark}

\subsubsection{First part of the bracketing argument}

Denote again $m=(n-r_{n})/2$ where $r_{n}=2[(\log n)/2]+1$.

\begin{lemma}
For $\mathcal{\mathring{E}}_{l,n}=\mathcal{\mathring{E}}_{m}\otimes
\mathcal{\mathring{E}}_{m}$ we have
\[
\mathcal{\mathring{E}}_{m}\otimes\mathcal{\mathring{E}}_{m}\approx
\mathcal{\mathring{E}}_{2m}\text{. }%
\]

\end{lemma}

\begin{proof}
Note that the measures in $\mathcal{\mathring{E}}_{m}\otimes\mathcal{\mathring
{E}}_{m}$ are product measures, which can be described, after a rearrangement
of components, as
\[
Q_{1,m}:=%
{\displaystyle\bigotimes\limits_{j=1}^{m}}
\left(  N(0,J_{j,m}(f))\otimes N(0,J_{j,m}(f))\right)
\]
whereas the measures in $\mathcal{\mathring{E}}_{2m}$ are
\[
Q_{2,m}:=%
{\displaystyle\bigotimes\limits_{j=1}^{m}}
\left(  N(0,J_{2j-1,2m}(f))\otimes N(0,J_{2j,2m}(f))\right)  .
\]
Now Lemma \ref{lem-helldist-covmatr} yields%
\[
H^{2}\left(  Q_{1,m},Q_{2,m}\right)  \leq C\sum_{j=1}^{m}\left(  \left(
J_{2j-1,2m}(f)-J_{j,m}(f)\right)  ^{2}+\left(  J_{2j,2m}(f)-J_{j,m}(f)\right)
^{2}\right)  .
\]
Define a partition of $(-\pi,\pi)$ into $n$ intervals $W_{j,n}$,
$j=1,\ldots,n$ of equal length and for any $f\in L_{2}(-\pi,\pi)$, let
\begin{equation}
\bar{f}_{n}=\sum_{j=1}^{n}J_{j,n}(f)\mathbf{1}_{W_{j,n}}\label{fbar-def}%
\end{equation}
be the $L_{2}$-projection of $f$ onto piecewise constant functions wrt the
partition. Note that we have
\[
\left\Vert \bar{f}_{2m}-\bar{f}_{m}\right\Vert _{2}^{2}=\frac{2\pi}{m}%
\sum_{j=1}^{m}\left(  \left(  J_{2j-1,2m}(f)-J_{j,m}(f)\right)  ^{2}+\left(
J_{2j,2m}(f)-J_{j,m}(f)\right)  ^{2}\right)
\]
so that
\[
H^{2}\left(  Q_{1,m},Q_{2,m}\right)  \leq Cm\left\Vert \bar{f}_{2m}-\bar
{f}_{m}\right\Vert _{2}^{2}\leq Cm\left(  \left\Vert f-\bar{f}_{2m}\right\Vert
_{2}^{2}+\left\Vert f-\bar{f}_{m}\right\Vert _{2}^{2}\right)  .
\]
The result now follows from%
\begin{equation}
\sup_{f\in\Sigma}m\left\Vert f-\bar{f}_{m}\right\Vert _{2}^{2}\rightarrow
0.\label{rate-in-sobolev}%
\end{equation}
which is a consequence of Lemmas \ref{lem-approx-piececonst} and
\ref{lem-period-imbed}.
\end{proof}

\subsubsection{Second part of the bracketing argument}

In view of $\mathcal{\mathring{E}}_{2m}=\mathcal{\mathring{E}}_{n-r_{n}}$, our
next aim is to show
\[
\mathcal{\mathring{E}}_{n-r_{n}}\approx\mathcal{\mathring{E}}_{n}%
\]
where $r_{n}$ does not grow too quickly. Previously we defined $r_{n}=2[(\log
n)/2]+1$, but we will assume more generally now that $r_{n}=o(n^{1/2})$.

Consider the gamma density with shape parameter $a>0$%
\[
g_{a}(x)=\frac{1}{\Gamma(a)}x^{a-1}\exp(-x)\text{, }x\geq0
\]
where $\Gamma(a)$ is the gamma function, and more generally the density with
additional scale parameter $s>0$
\[
g_{a,s}(x)=\frac{1}{\Gamma(a)}s^{-a}x^{a-1}\exp(-xs^{-1})\text{, }x\geq0.
\]
We will call the respective law the $\Gamma(a,s)$ law. Clearly if $X\sim
\Gamma(a,1)$ then $sX\sim\Gamma(a,s)$. It is well known that $\Gamma
(n/2,2)=\chi_{n}^{2}$ and that the following result holds. Assume $X\sim
\Gamma(a,s)$ and $Y\sim\Gamma(b,s)$; then $X+Y$, $X/(X+Y)$ are independent
random variables, and $X+Y\sim\Gamma(a+b,s)$ while $X/(X+Y)$ has a Beta$(a,b)$
distribution (Bickel and Doksum (2001), Theorem B.2.3, p. 489).

Furthermore, for fixed $a>0$ consider the family of laws
\begin{equation}
\left(  \Gamma(a,s),s>0\right)  .\label{gamma-family}%
\end{equation}
Clearly this is a one parameter exponential family; the shape of this
exponential family implies that in a product family
\[
\left(  \Gamma^{\otimes n}(a,s),s>0\right)
\]
with $n$ i.i.d. observations $X_{1},\ldots,X_{n}$, the sum $\sum_{i=1}%
^{n}X_{i}$ is a sufficient statistic. This sufficient statistic has law
$\Gamma(na,s)$; hence for any subset $S\subset(0,\infty)$ we have the
equivalence of experiments
\begin{equation}
\left(  \Gamma^{\otimes n}(a,s),s\in S\right)  \sim\left(  \Gamma(na,s),s\in
S\right)  .\label{gamma-suffic-arg}%
\end{equation}

\begin{lemma}
\label{lem-helldist-gamma-sameshape}For all $a>0$ and for $s,t>0$
\[
H^{2}\left(  \Gamma(a,s),\Gamma(a,t)\right)  =2\left(  1-\left(
1-\frac{\left(  s^{1/2}-t^{1/2}\right)  ^{2}}{s+t}\right)  ^{a}\right)  .
\]

\end{lemma}

\begin{proof}
We have
\[
H^{2}\left(  \Gamma(a,s),\Gamma(a,t)\right)  =2\left(  1-\int g_{a,s}%
^{1/2}(x)g_{a,t}^{1/2}(x)dx\right)  ,
\]%
\[
\int g_{a,s}^{1/2}(x)g_{a,t}^{1/2}(x)dx=\frac{1}{\Gamma(a)}\int_{0}^{\infty
}x^{a-1}s^{-a/2}t^{-a/2}\exp\left(  -x\left(  \frac{1}{2s}+\frac{1}%
{2t}\right)  \right)  dx.
\]
With a substitution $u=x\left(  \frac{1}{2s}+\frac{1}{2t}\right)  $ this
becomes
\begin{align*}
& \frac{1}{\Gamma(a)}\int_{0}^{\infty}\left(  \frac{2s^{1/2}t^{1/2}}%
{s+t}\right)  ^{a}u^{a-1}\exp(-u)du\\
& =\left(  \frac{2s^{1/2}t^{1/2}}{s+t}\right)  ^{a}=\left(  1-\frac{\left(
s^{1/2}-t^{1/2}\right)  ^{2}}{s+t}\right)  ^{a}.
\end{align*}

\end{proof}

\begin{lemma}
\label{lem-helldist-gamma-samescale}We have, for all $s>0$ and $a,b>0$%
\[
H^{2}\left(  \Gamma(a,s),\Gamma(b,s)\right)  =2\left(  1-\frac{\Gamma
((a+b)/2)}{\left(  \Gamma(a)\Gamma(b)\right)  ^{1/2}}\right)  .
\]

\end{lemma}

\begin{proof}
In this case
\begin{align*}
\int g_{a,s}^{1/2}(x)g_{b,s}^{1/2}(x)dx  & =\frac{1}{\Gamma^{1/2}%
(a)\Gamma^{1/2}(b)}\int_{0}^{\infty}x^{(a+b)/2-1}s^{-(a+b)/2}\exp\left(
-xs^{-1}\right)  dx\\
& =\frac{\Gamma((a+b)/2)}{\Gamma^{1/2}(a)\Gamma^{1/2}(b)}.
\end{align*}

\end{proof}

In $\mathcal{\mathring{E}}_{n}$ we observe (cp. (\ref{exp-e-dot-def})
\[
z_{j}=J_{j,n}^{1/2}(f)\xi_{j},\;j=1,\ldots,n
\]
for independent standard normals $\xi_{j}$, which by sufficiency is equivalent
to observing $z_{j}^{2}=J_{j,n}(f)\xi_{j}^{2}$. Thus $\mathcal{\mathring{E}%
}_{n}$ is equivalent to
\begin{equation}
\mathcal{\mathring{E}}_{n,1}:=\left(
{\displaystyle\bigotimes\limits_{j=1}^{n}}
\Gamma(1/2,2J_{j,n}(f)),\;f\in\Sigma\right)  .\label{exper-e-dot-one-def}%
\end{equation}
Set again $m=n-r_{n}$. The above experiment in turn is equivalent, by the
sufficiency argument for the scaled gamma law invoked in
(\ref{gamma-suffic-arg}), to
\[
\mathcal{\mathring{E}}_{n,m}:=\left(
{\displaystyle\bigotimes\limits_{j=1}^{n}}
\Gamma^{\otimes m}(1/2m,2J_{j,n}(f)),\;f\in\Sigma\right)  .
\]
Analogously we have
\begin{equation}
\mathcal{\mathring{E}}_{m}\sim\mathcal{\mathring{E}}_{m,1}\sim
\mathcal{\mathring{E}}_{m,n}:=\left(
{\displaystyle\bigotimes\limits_{j=1}^{m}}
\Gamma^{\otimes n}(1/2n,2J_{j,m}(f)),\;f\in\Sigma\right)
.\label{triple-equiv}%
\end{equation}
Introduce an intermediate experiment
\[
\mathcal{\mathring{E}}_{m,n}^{\ast}:=\left(
{\displaystyle\bigotimes\limits_{j=1}^{m}}
\Gamma^{\otimes n}(1/2m,2J_{j,m}(f)),\;f\in\Sigma\right)
\]

\begin{lemma}
We have the total variation asymptotic equivalence
\[
\mathcal{\mathring{E}}_{m,n}^{\ast}\simeq\mathcal{\mathring{E}}_{n,m}\text{ as
}n\rightarrow\infty\text{.}%
\]

\end{lemma}

\begin{proof}
Write the measures in $\mathcal{\mathring{E}}_{n,m}$ as a product of $mn$
components, i.e. as $\otimes_{i=1}^{mn}Q_{1,i}$ where the component measures
$Q_{1,i}$ are defined as follows. For every $i=1,\ldots,mn$, let $j(1,i)$. be
the unique index $j\in\left\{  1,\ldots,n\right\}  $ such that there exists
$k\in\left\{  1,\ldots,m\right\}  $ for which $i=(j-1)m+k$. Then
\[
Q_{1,i}:=\Gamma(1/2m,2J_{j(1,i),n}(f))\text{, }i=1,\ldots,mn\text{. }%
\]
Analogously, let $j(2,i)$ be the unique index $j\in\left\{  1,\ldots
,m\right\}  $ such that there exists $k\in\left\{  1,\ldots,n\right\}  $ for
which $i=(j-1)n+k$. Then the measures in $\mathcal{\mathring{E}}_{m,n}^{\ast}$
are a product of $mn$ components, i. e. are $\otimes_{i=1}^{mn}Q_{2,i}$ where
\[
Q_{2,i}=\Gamma(1/2m,2J_{j(2,i),m}(f))\text{, }i=1,\ldots,mn.\text{ }%
\]
Then the Hellinger distance between measures in $\mathcal{\mathring{E}}_{n,m}
$ and $\mathcal{\mathring{E}}_{m,n}^{\ast}$ is, using Lemma 2.19 in Strasser
(1985) and then Lemma \ref{lem-helldist-gamma-sameshape}
\begin{equation}
H^{2}\left(
{\displaystyle\bigotimes\limits_{i=1}^{mn}}
Q_{1,i},%
{\displaystyle\bigotimes\limits_{i=1}^{mn}}
Q_{2,i}\right)  \leq2\sum_{i=1}^{mn}H^{2}\left(  Q_{1,i},Q_{2,i}\right)
\label{prod-gammmas}%
\end{equation}%
\[
=4\sum_{i=1}^{mn}\left(  1-\left(  1-\frac{\left(  J_{j(1,i),n}^{1/2}%
(f)-J_{j(2,i),m}^{1/2}(f)\right)  ^{2}}{J_{j(1,i),n}(f)+J_{j(2,i),m}%
(f)}\right)  ^{1/2m}\right)  .
\]
By using the inequality
\[
\frac{\left(  s^{1/2}-t^{1/2}\right)  ^{2}}{s+t}=\frac{\left(  s-t\right)
^{2}}{(s+t)(s^{1/2}+t^{1/2})^{2}}\leq\frac{\left(  s-t\right)  ^{2}}{s^{2}}%
\]
and observing that for $f\in\Sigma$, we have $J_{j,n}(f)\geq M^{-1}$, we
obtain an upper bound for (\ref{prod-gammmas})
\begin{equation}
4\sum_{i=1}^{mn}\left(  1-\left(  1-M^{2}\left(  J_{j(1,i),n}(f)-J_{j(2,i),m}%
(f)\right)  ^{2}\right)  ^{1/2m}\right)  .\label{upp-bound-1}%
\end{equation}
The expression $J_{j(1,i),n}(f)-J_{j(2,i),m}(f)$ can be described as follows.
For any $x\in\left(  \frac{i-1}{mn},\frac{i}{mn}\right)  $, $i=1,\ldots,mn$ we
have
\begin{equation}
J_{j(1,i),n}(f)-J_{j(2,i),m}(f)=\bar{f}_{n}(x)-\bar{f}_{m}%
(x).\label{averages-diiferent}%
\end{equation}
where $\bar{f}_{n}$ is defined by (\ref{fbar-def}). Now as a consequence of
Lemmas \ref{lem-unif-approx-piececonst} and \ref{lem-period-imbed}
\begin{equation}
\sup_{f\in\Sigma}\left\Vert \bar{f}_{n}-\bar{f}_{m}\right\Vert _{\infty}%
\leq\sup_{f\in\Sigma}\left\Vert f-\bar{f}_{n}\right\Vert _{\infty}+\sup
_{f\in\Sigma}\left\Vert f-\bar{f}_{m}\right\Vert _{\infty}%
=o(1).\label{different-piecewi-sup-norm}%
\end{equation}
Note that for $m\rightarrow\infty$ and $z\rightarrow0$ we have%
\[
\left(  1-Cz^{2}\right)  ^{1/2m}=\exp\left(  \frac{1}{2m}\log\left(
1-Cz^{2}\right)  \right)
\]%
\[
=\exp\left(  -\frac{1}{2m}\left(  Cz^{2}+O(z^{4})\right)  \right)  =1-\frac
{1}{2m}\left(  Cz^{2}+O(z^{4})\right)  +o\left(  \frac{z^{2}}{m}\right)  .
\]
Thus from (\ref{upp-bound-1}) we obtain in view of
(\ref{different-piecewi-sup-norm})
\[
H^{2}\left(
{\displaystyle\bigotimes\limits_{i=1}^{mn}}
Q_{1,i},%
{\displaystyle\bigotimes\limits_{i=1}^{mn}}
Q_{2,i}\right)  \leq C\sum_{i=1}^{mn}\frac{1}{m}\left(  J_{j(1,i),n}%
(f)-J_{j(2,i),m}(f)\right)  ^{2}(1+o(1)).
\]
As a consequence of (\ref{averages-diiferent}) we obtain
\[
\left\Vert \bar{f}_{n}-\bar{f}_{m}\right\Vert _{2}^{2}=\sum_{i=1}^{mn}\frac
{1}{mn}\left(  J_{j(1,i),n}(f)-J_{j(2,i),m}(f)\right)  ^{2}%
\]
which implies%
\[
H^{2}\left(
{\displaystyle\bigotimes\limits_{i=1}^{mn}}
Q_{1,i},%
{\displaystyle\bigotimes\limits_{i=1}^{mn}}
Q_{2,i}\right)  \leq Cn\left\Vert \bar{f}_{n}-\bar{f}_{m}\right\Vert _{2}%
^{2}\leq Cn\left\Vert f-\bar{f}_{m}\right\Vert _{2}^{2}+Cn\left\Vert f-\bar
{f}_{n}\right\Vert _{2}^{2}.
\]
Now as in (\ref{rate-in-sobolev}) this upper bound is $o(1)$ uniformly over
$f\in\Sigma$.
\end{proof}

\begin{lemma}
We have the asymptotic equivalence
\[
\mathcal{\mathring{E}}_{m,n}^{\ast}\simeq\mathcal{\mathring{E}}_{m,n}\text{ as
}n\rightarrow\infty\text{.}%
\]

\end{lemma}

\begin{proof}
We know (cf. (\ref{triple-equiv}), (\ref{exper-e-dot-one-def})) that
$\mathcal{\mathring{E}}_{m,n}\sim\mathcal{\mathring{E}}_{m,1}$ where
\[
\mathcal{\mathring{E}}_{m,1}=\left(
{\displaystyle\bigotimes\limits_{j=1}^{m}}
\Gamma(1/2,2J_{j,m}(f)),\;f\in\Sigma\right)  .
\]
Analogously, using (\ref{gamma-suffic-arg}) again, we obtain%
\[
\mathcal{\mathring{E}}_{m,n}^{\ast}\sim\overline{\mathcal{E}}_{m,1}^{\ast
}:=\left(
{\displaystyle\bigotimes\limits_{j=1}^{m}}
\Gamma(n/2m,2J_{j,m}(f)),\;f\in\Sigma\right)
\]
For given $f\in\Sigma$, the Hellinger distance between the two respective
product measures is bounded by (using Lemma 2.19 in Strasser (1985) and then
Lemma \ref{lem-helldist-gamma-samescale})%
\[
2\sum_{j=1}^{m}H^{2}\left(  \Gamma(1/2,2J_{j,m}(f)),\Gamma(n/2m,2J_{j,m}%
(f))\right)  =4\sum_{j=1}^{m}\left(  1-\frac{\Gamma(1/4+n/4m)}{\left(
\Gamma(1/2)\Gamma(n/2m)\right)  ^{1/2}}\right)  .
\]
Note that this bound does not depend on $f\in\Sigma$. Write $n/m=1+\delta$
where $\delta=r_{n}/m$; the above is
\begin{equation}
4\sum_{j=1}^{m}\frac{\left(  \Gamma(1/2)\Gamma(1/2+\delta/2)\right)
^{1/2}-\Gamma(1/2+\delta/4)}{\left(  \Gamma(1/2)\Gamma(1/2+\delta/2)\right)
^{1/2}}.\label{gamma-expression-a}%
\end{equation}
The Gamma function is infinitely differentiable on $(0,\infty)$; by a Taylor
expansion we obtain%
\begin{align*}
\Gamma(1/2+\delta/4)  & =\Gamma(1/2)+\Gamma^{\prime}(1/2)\frac{\delta}%
{4}+O(\delta^{2}),\\
\Gamma^{1/2}(1/2+\delta/2)  & =\Gamma^{1/2}(1/2)+\frac{1}{2}\Gamma
^{-1/2}(1/2)\Gamma^{\prime}(1/2)\frac{\delta}{2}+O(\delta^{2}).
\end{align*}
Consequently
\[
\left(  \Gamma(1/2)\Gamma(1/2+\delta/2)\right)  ^{1/2}-\Gamma(1/2+\delta
/4)=O(\delta^{2})
\]
so that (\ref{gamma-expression-a}) becomes
\[
\sum_{j=1}^{m}\frac{O(\delta^{2})}{\Gamma(1/2)(1+o(1))}=m\delta^{2}%
O(1)\leq\frac{r_{n}^{2}}{m}O(1).
\]
The condition $r_{n}=o(n^{1/2})$ now implies that this upper bound is $o(1)$.
We thus established total variation asymptotic equivalence $\mathcal{\mathring
{E}}_{m,1}\simeq\overline{\mathcal{E}}_{m,1}^{\ast}$.
\end{proof}

\section{Appendix: auxiliary statements and analytic facts \label{sec-proofs}}

\subsection{Proof of Lemma \ref{lem-helldist-covmatr}}

Consider the\textbf{\ }spectral decompositions of $A$ and $B$:
\[
A=C_{1}\Lambda_{1}C_{1}^{\prime}\text{, }B=C_{2}\Lambda_{2}C_{2}^{\prime}%
\]
where $\Lambda_{i}$ are $n\times n$ diagonal matrices and $C_{i}$ are
orthogonal matrices. Recall the simultaneous diagonalization of $A$ and $B$:
setting $D=\Lambda_{1}^{-1/2}C_{1}^{\prime}$, we obtain
\[
DAD^{\prime}=I_{n}\text{, }\tilde{B}:=DBD^{\prime}%
\]
and letting $\tilde{B}=\tilde{C}\tilde{\Lambda}\tilde{C}^{\prime}$ be the
spectral decomposition of $\tilde{B}$, we obtain with $\tilde{D}:=\tilde
{C}^{\prime}D$%
\[
\tilde{D}A\tilde{D}^{\prime}=I_{n},\tilde{D}B\tilde{D}^{\prime}=\tilde
{\Lambda}\mathbf{.}%
\]
We now claim that
\begin{equation}
\left\Vert I_{n}-\tilde{\Lambda}\right\Vert ^{2}\leq M^{2}\left\Vert
A-B\right\Vert ^{2}.\label{inequ-squarednorm-a}%
\end{equation}
Indeed we have
\begin{align*}
\left\Vert I_{n}-\tilde{\Lambda}\right\Vert ^{2}  & =\mathrm{tr}\left[
(I_{n}-\tilde{\Lambda})(I_{n}-\tilde{\Lambda}\mathbf{)}\right] \\
& =\mathrm{tr}\left[  \left(  \tilde{D}A\tilde{D}^{\prime}-\tilde{D}B\tilde
{D}^{\prime}\right)  \left(  \tilde{D}A\tilde{D}^{\prime}-\tilde{D}B\tilde
{D}^{\prime}\right)  \right] \\
& =\mathrm{tr}\left[  \tilde{D}^{\prime}\tilde{D}\left(  A-B\right)  \tilde
{D}^{\prime}\tilde{D}\left(  A-B\right)  \right]
\end{align*}
Now for eigenvalues $\lambda_{\max}(\cdot)$ we have
\begin{align}
\lambda_{\max}\left(  \tilde{D}^{\prime}\tilde{D}\right)   & =\lambda_{\max
}\left(  \tilde{D}\tilde{D}^{\prime}\right)  =\lambda_{\max}\left(  \tilde
{C}^{\prime}DD^{\prime}\tilde{C}\right)  =\lambda_{\max}\left(  DD^{\prime
}\right) \nonumber\\
& =\lambda_{\max}\left(  C_{1}\Lambda_{1}^{-1/2}\Lambda_{1}^{-1/2}%
C_{1}^{\prime}\right)  =\lambda_{\min}^{-1}\left(  A\right)  \leq
M,\label{eigenva-estim-1-aa}%
\end{align}
hence
\begin{align*}
\left\Vert I_{n}-\tilde{\Lambda}\right\Vert ^{2}  & \leq M\;\mathrm{tr}\left[
\tilde{D}^{\prime}\tilde{D}\left(  A-B\right)  \left(  A-B\right)  \right] \\
& \leq M^{2}\;\mathrm{tr}\left[  \left(  A-B\right)  \left(  A-B\right)
\right]  =M^{2}\left\Vert A-B\right\Vert ^{2}%
\end{align*}
so that (\ref{inequ-squarednorm-a}) is proved. Similarly to
(\ref{eigenva-estim-1-aa}) we obtain a bound from below
\[
\lambda_{\min}\left(  \tilde{D}^{\prime}\tilde{D}\right)  =\lambda_{\max}%
^{-1}\left(  A\right)  \geq M^{-1}%
\]
which yields analogously to (\ref{inequ-squarednorm-a})
\begin{equation}
\left\Vert I_{n}-\tilde{\Lambda}\right\Vert ^{2}\geq M^{-2}\left\Vert
A-B\right\Vert ^{2}.\label{inequ-squarednorm-2-a}%
\end{equation}
Consider now the Hellinger affinity $A_{H}(\cdot,\cdot)$ between the one
dimensional normals $N(0,1)$ and $N(0,\sigma^{2})$: if $\varphi$ is the
standard normal density then
\begin{align*}
A_{H}(N(0,1),N(0,\sigma^{2}))  & =\sigma^{-1/2}\int\varphi^{1/2}%
(x)\varphi^{1/2}(x\sigma^{-1})dx\\
& =\left(  \frac{2\sigma}{1+\sigma^{2}}\right)  ^{1/2}=\left(  1-\frac
{(1-\sigma)^{2}}{1+\sigma^{2}}\right)  ^{1/2}\\
& =\left(  1-\frac{(1-\sigma^{2})^{2}}{(1+\sigma)^{2}(1+\sigma^{2})}\right)
^{1/2}%
\end{align*}
Let $h=\sigma^{2}-1$; then as $h\rightarrow0$%
\begin{equation}
\log A_{H}(N(0,1),N(0,\sigma^{2}))=-\frac{h^{2}}{16}%
(1+o(1)).\label{log-affi-a}%
\end{equation}
The matrix $\tilde{D}$ is nonsingular, and since the Hellinger distance is
invariant under one-to-one transformations,
\[
H^{2}\left(  N_{n}(0,A),N_{n}(0,B)\right)  =H^{2}\left(  N_{n}(0,I_{n}%
),N_{n}(0,\tilde{\Lambda})\right)  =
\]%
\begin{align}
& =2\left(  1-A_{H}(N_{n}(0,I_{n}),N_{n}(0,\tilde{\Lambda}))\right)  =2\left(
1-\prod_{i=1}^{n}A_{H}(N(0,1),N(0,\tilde{\lambda}_{i}))\right) \nonumber\\
& =2\left(  1-\exp\left(  \sum_{i=1}^{n}\log\left(  A_{H}(N(0,1),N(0,\tilde
{\lambda}_{i}))\right)  \right)  \right) \label{from-consequence}%
\end{align}
where $\tilde{\lambda}_{i}$, $i=1,\ldots,n$ are the diagonal elements of
$\tilde{\Lambda}$. Let us assume that $\left\Vert A-B\right\Vert \leq
\epsilon\rightarrow0$ where the dimension $n$ of $A,B$ may vary arbitrarily.
Since
\[
\sup_{i=1,\ldots,n}\left\vert 1-\tilde{\lambda}_{i}\right\vert ^{2}\leq
\sum_{i=1}^{n}(1-\tilde{\lambda}_{i})^{2}=\left\Vert I_{n}-\tilde{\Lambda
}\right\Vert ^{2}%
\]
we may write, in view of (\ref{inequ-squarednorm-a}) and (\ref{log-affi-a})
\[
\log\left(  A_{H}(N(0,1),N(0,\tilde{\lambda}_{i}))\right)  =-\frac{1}%
{16}\left(  1-\tilde{\lambda}_{i}\right)  ^{2}(1+\rho_{i})
\]
where $\sup_{i=1,\ldots,n}\left\vert \rho_{i}\right\vert \rightarrow0$ as
$\epsilon\rightarrow0$. Since
\[
\left\vert \sum_{i=1}^{n}\left(  1-\tilde{\lambda}_{i}\right)  ^{2}\rho
_{i}\right\vert \leq\left(  \sup_{i=1,\ldots,n}\left\vert \rho_{i}\right\vert
\right)  \sum_{i=1}^{n}\left(  1-\tilde{\lambda}_{i}\right)  ^{2}%
\]
we obtain for $\epsilon\rightarrow0$
\[
-16\sum_{i=1}^{n}\log\left(  A_{H}(N(0,1),N(0,\tilde{\lambda}_{i}))\right)
=\sum_{i=1}^{n}\left(  1-\tilde{\lambda}_{i}\right)  ^{2}(1+\rho_{i})
\]%
\[
=\sum_{i=1}^{n}\left(  1-\tilde{\lambda}_{i}\right)  ^{2}+o(1)\sum_{i=1}%
^{n}\left(  1-\tilde{\lambda}_{i}\right)  ^{2}=\left\Vert I_{n}-\tilde
{\Lambda}\right\Vert ^{2}(1+o(1))
\]
and as a consequence from (\ref{from-consequence})
\begin{align*}
H^{2}\left(  N_{n}(0,A),N_{n}(0,B)\right)   & =2\left(  1-\exp\left(
-\frac{1}{16}\left\Vert I_{n}-\tilde{\Lambda}\right\Vert ^{2}(1+o(1))\right)
\right) \\
& =\frac{1}{8}\left\Vert I_{n}-\tilde{\Lambda}\right\Vert ^{2}(1+o(1)).
\end{align*}
In conjunction with (\ref{inequ-squarednorm-a}) and
(\ref{inequ-squarednorm-2-a}), the last relation proves the lemma.

\subsection{An auxiliary result for the proof of Lemma \ref{lem-local-likproc}%
}

Let $A,B$ be two $n\times n$ covariance matrices. Recall that for every
covariance matrix $A$ there is a uniquely defined symmetric square root matrix
$A^{1/2}$: if $A=D\Lambda D^{\top}$ is a spectral decomposition ($D$
orthogonal, $\Lambda$ diagonal) of $A$ then $A^{1/2}=D\Lambda^{1/2}D^{\top}$.

\begin{lemma}
\label{lem-sqrt-matr}Let $A\mathbf{,}B$ be two $n\times n$ covariance
matrices. Then
\[
\left\Vert {\normalsize A}^{1/2}-{\normalsize B}^{1/2}\right\Vert
\lambda_{\min}({\normalsize A}^{1/2}+{\normalsize B}^{1/2})\leq\left\Vert
{\normalsize A}-{\normalsize B}\right\Vert .
\]

\end{lemma}

\begin{proof}
Observe that
\begin{align*}
\left(  {\normalsize A}^{1/2}-{\normalsize B}^{1/2}\right)  {\normalsize B}%
^{1/2}+{\normalsize A}^{1/2}\left(  {\normalsize A}^{1/2}-{\normalsize B}%
^{1/2}\right)   & ={\normalsize A}-{\normalsize B}\\
\left(  {\normalsize A}^{1/2}-{\normalsize B}^{1/2}\right)  {\normalsize A}%
^{1/2}+{\normalsize B}^{1/2}\left(  {\normalsize A}^{1/2}-{\normalsize B}%
^{1/2}\right)   & ={\normalsize A}-{\normalsize B}%
\end{align*}
Add up the two equations and set $S\mathbf{=}\left(  A^{1/2}+B^{1/2}\right)
$; $D=\left(  A^{1/2}-B^{1/2}\right)  $; then
\begin{equation}
{\normalsize DS}+{\normalsize SD}=2\left(  {\normalsize A}-{\normalsize B}%
\right)  \mathbf{.}\label{key-matrx-id}%
\end{equation}
Take the squared norm $\left\Vert \cdot\right\Vert ^{2}\mathbf{\ }$on both
sides and observe
\begin{align*}
\left\Vert {\normalsize DS}+{\normalsize SD}\right\Vert ^{2}  & =\mathrm{tr}%
\left[  {\normalsize (DS+SD)(DS+SD)}\right] \\
& =2\mathrm{tr}\left[  {\normalsize DSSD}\right]  +2\mathrm{tr}\left[
{\normalsize DSDS}\right]  .
\end{align*}
Clearly we have
\begin{align*}
\mathrm{tr}\left[  {\normalsize DSSD}\right]   & \geq\left(  \lambda_{\min
}({\normalsize S})\right)  ^{2}\mathrm{tr}\left[  {\normalsize DD}\right] \\
\mathrm{tr}\left[  {\normalsize DSDS}\right]   & =\mathrm{tr}\left[
{\normalsize S}^{1/2}{\normalsize DSDS}^{1/2}\right]  \geq\lambda_{\min
}({\normalsize S})\mathrm{tr}\left[  {\normalsize S}^{1/2}{\normalsize DDS}%
^{1/2}\right] \\
& \geq\left(  \lambda_{\min}({\normalsize S})\right)  ^{2}\mathrm{tr}\left[
{\normalsize DD}\right]  .
\end{align*}
The last two displays imply
\[
\left\Vert {\normalsize DS+SD}\right\Vert ^{2}\geq4\left(  \lambda_{\min
}({\normalsize S})\right)  ^{2}\left\Vert D\right\Vert ^{2}%
\]
which in conjunction with (\ref{key-matrx-id}) yields
\[
\left\Vert {\normalsize D}\right\Vert \lambda_{\min}({\normalsize S}%
)\leq\left\Vert {\normalsize A-B}\right\Vert .
\]

\end{proof}

\subsection{Besov spaces on an interval\label{subsec-besov}}

Let $f$ be a function defined on $I=[0,1]$ \ and for $0<h<1$ define
\begin{align*}
\left\Vert \Delta_{h}f\right\Vert _{p}^{p}  & :=\int_{0}^{1-h}\left\vert
f(x)-f(x+h)\right\vert ^{p}dx\text{ for }1\leq p<\infty\text{,}\\
\left\Vert \Delta_{h}f\right\Vert _{\infty}  & =\sup_{0\leq x\leq
1-h}\left\vert f(x)-f(x+h)\right\vert .
\end{align*}
For $1\leq p\leq\infty$, the modulus of smoothness is defined as
\[
\omega(f,t)_{p}:=\sup_{0<h\leq t}\left\Vert \Delta_{h}f\right\Vert _{p}.
\]
For $0<\alpha<1$ and $1\leq q\leq\infty$ define a Besov type seminorm
$\left\vert f\right\vert _{B_{p,q}^{\alpha}}$ by
\begin{align*}
\left\vert f\right\vert _{B_{p,q}^{\alpha}}  & :=\left(  \int_{0}^{\infty
}\left(  \frac{\omega(f,t)_{p}}{t^{\alpha}}\right)  ^{q}\frac{dt}{t}\right)
^{1/q}\text{ for }1\leq q<\infty,\\
\left\vert f\right\vert _{B_{p,q}^{\alpha}}  & =\sup_{t>1}\omega
(f,t)_{p}\text{ for }q=\infty
\end{align*}
and a norm $\left\Vert f\right\Vert _{B_{p,q}^{\alpha}}$ by
\[
\left\Vert f\right\Vert _{B_{p,q}^{\alpha}}:=\left\Vert f\right\Vert
_{p}+\left\vert f\right\vert _{B_{p,q}^{\alpha}}.
\]
The Besov space $B_{p,q}^{\alpha}$ (for $1\leq p<\infty$, $0<\alpha<1$) is the
set of $f$ where $\left\Vert f\right\Vert _{B_{p,q}^{\alpha}}<\infty$,
equipped with the norm $\left\Vert \cdot\right\Vert _{B_{p,q}^{\alpha}}$.
Define also the H\"{o}lder norm
\begin{equation}
\left\Vert f\right\Vert _{C^{\alpha}}:=\left\Vert f\right\Vert _{\infty}%
+\sup_{x\neq y}\frac{\left\vert f(x)-f(y)\right\vert }{\left\vert
x-y\right\vert }\label{holder-norm}%
\end{equation}
and the corresponding H\"{o}lder space $C^{\alpha}$. For two different spaces,
$B$ and $B^{\prime}$ say, an \textit{embedding theorem }(written
$B\hookrightarrow B^{\prime}$) is a norm inequality
\[
\left\Vert f\right\Vert _{B^{\prime}}\leq C\left\Vert f\right\Vert _{B}%
\]
where $C$ depends on $B^{\prime}$, $B$. Thus the embedding implies the set
inclusion $B\subset B^{\prime}$. We cite the basic embedding theorem for our
case, which is obtained by combining Theorems 18.4 , 18.5 18.8 in Besov, Il'in
and Nikol'skii (1979) with Theorems 3.3.1 and 2.5.7. in Triebel (1983), for
the special case of a domain $[0,1]$.

\begin{proposition}
\label{prop-embed-theor}Let $0<\alpha^{\prime}<\alpha<1$ and $1\leq
p,q,p^{\prime},q^{\prime}\leq\infty.$ Then \newline i) if $q<q^{\prime}$ then
\[
B_{p,q}^{\alpha}\hookrightarrow B_{p,q^{\prime}}^{\alpha}%
\]
ii) if $p<p^{\prime}$ and $\alpha-\left(  1/p-1/p^{\prime}\right)  >0$ then
for $\alpha^{\prime}=\alpha-\left(  1/p-1/p^{\prime}\right)  $
\[
B_{p,q}^{\alpha}\hookrightarrow B_{p^{\prime},q}^{\alpha^{\prime}}%
\]
iii) if $p\geq p^{\prime}$ then
\[
B_{p,q}^{\alpha}\hookrightarrow B_{p^{\prime},q^{\prime}}^{\alpha^{\prime}}.
\]
iv) we have $B_{\infty,\infty}^{\alpha}=C^{\alpha}$ in the sense of
equivalence of norms:%
\[
B_{\infty,\infty}^{\alpha}\hookrightarrow C^{\alpha}\text{ and }C^{\alpha
}\hookrightarrow B_{\infty,\infty}^{\alpha}.
\]

\end{proposition}

\textbf{Approximation by step functions. }Consider a partition of $[0,1]$ into
$n$ intervals $W_{j,n}$ $j=1,\ldots,n$ of equal length and for any $f\in
L_{2}(0,1),$ let $\bar{f}_{n}$ be the $L_{2}$-projection onto the piecewise
constant functions, i.e.
\begin{equation}
\bar{f}_{n}=\sum_{j=1}^{n}J_{j,n}(f)\mathbf{1}_{W_{j,n}}\text{, where }%
J_{j,n}(f)=n\int_{W_{jn,}}f(x)dx.\label{confer-intervals}%
\end{equation}

\begin{lemma}
\label{lem-approx-piececonst}For $0<\alpha<1$ and $f\in B_{2,2}^{\alpha}$ we
have
\[
\left\Vert f-\bar{f}_{n}\right\Vert _{2}^{2}\leq4\;n^{-2\alpha}\;\left\vert
f\right\vert _{B_{2,2}^{\alpha}}^{2}.
\]

\end{lemma}

\begin{proof}
Note first
\begin{equation}
\left\Vert f-\bar{f}_{n}\right\Vert _{2}^{2}=\sum_{j=1}^{n}\int_{W_{j,n}%
}\left(  f(x)-J_{j,n}(f)\right)  ^{2}dx.\label{addit-1}%
\end{equation}
For any interval $(a,b)$ and $\varepsilon=b-a$ we have
\begin{align*}
\int_{a}^{b}\left(  f(x)-\varepsilon^{-1}\int_{a}^{b}f(u)du\right)  ^{2}dx  &
=\int_{a}^{b}\left(  \varepsilon^{-1}\int_{a}^{b}\left(  f(x)-f(u)\right)
du\right)  ^{2}dx\\
\text{(by Jensen's inequality)}  & \leq\int_{a}^{b}\varepsilon^{-1}\int
_{a}^{b}\left(  f(x)-f(u)\right)  ^{2}dudx\\
& =2\varepsilon^{-1}\int_{a}^{b}\int_{x}^{b}\left(  f(x)-f(u)\right)
^{2}dudx.
\end{align*}
With a change of variable $h=u-x$ the above equals%
\begin{align*}
& 2\varepsilon^{-1}\int_{a}^{b}\int_{0}^{b-x}\left(  f(x)-f(x+h)\right)
^{2}dhdx\\
& \leq2\varepsilon^{-1}\int_{a}^{b}\int_{0}^{b-x}(b-x)^{2\alpha+1}%
\frac{\left(  f(x)-f(x+h)\right)  ^{2}}{h^{2\alpha+1}}dhdx\\
& \leq2\varepsilon^{2\alpha}\int_{a}^{b}\int_{0}^{b-x}\frac{\left(
f(x)-f(x+h)\right)  ^{2}}{h^{2\alpha+1}}dhdx.
\end{align*}
Setting now $(a,b)=W_{j,n}$, $b=j/n$ and $\varepsilon=n^{-1}$ we obtain for
$j=1,\ldots,n-1$
\begin{align}
\int_{W_{j,n}}\left(  f(x)-J_{j,n}(f)\right)  ^{2}dx  & \leq2n^{-2\alpha
}\;\int_{W_{j,n}}\int_{0}^{j/n-x}\frac{\left(  f(x)-f(x+h)\right)  ^{2}%
}{h^{2\alpha+1}}dhdx\label{onlybound}\\
& \leq2n^{-2\alpha}\;\int_{W_{j,n}}\int_{0}^{1/n}\frac{\left(
f(x)-f(x+h)\right)  ^{2}}{h^{2\alpha+1}}dhdx\nonumber
\end{align}
whereas for $j=n$ we have only the bound (\ref{onlybound}), i.e
\[
\int_{W_{n,n}}\left(  f(x)-J_{n,n}(f)\right)  ^{2}dx\leq2n^{-2\alpha}%
\;\int_{W_{n,n}}\int_{0}^{1-x}\frac{\left(  f(x)-f(x+h)\right)  ^{2}%
}{h^{2\alpha+1}}dhdx.
\]
Hence from (\ref{addit-1}) by adding the upper bounds%
\[
\left\Vert f-\bar{f}_{n}\right\Vert _{2}^{2}=\sum_{j=1}^{n-1}\int_{W_{j,n}%
}\left(  f(x)-J_{j,n}(f)\right)  ^{2}dx+\int_{W_{n,n}}\left(  f(x)-J_{n,n}%
(f)\right)  ^{2}dx,
\]%
\begin{align*}
\sum_{j=1}^{n-1}\int_{W_{j,n}}\left(  f(x)-J_{j,n}(f)\right)  ^{2}dx  &
\leq2n^{-2\alpha}\int_{0}^{1/n}\int_{0}^{1-1/n}\frac{\left(
f(x)-f(x+h)\right)  ^{2}}{h^{2\alpha+1}}dxdh\\
& \leq2n^{-2\alpha}\;\int_{0}^{1/n}\left\Vert \Delta_{h}f\right\Vert _{2}%
^{2}h^{-(2\alpha+1)}dh
\end{align*}
and
\begin{equation}
\int_{W_{n,n}}\left(  f(x)-J_{n,n}(f)\right)  ^{2}dx\leq2n^{-2\alpha}%
\int_{1-1/n}^{1}\int_{0}^{1-x}\frac{\left(  f(x)-f(x+h)\right)  ^{2}%
}{h^{2\alpha+1}}dhdx.\label{secnd-integral-a}%
\end{equation}
Now set
\begin{align*}
g(x,h)  & =\left(  f(x)-f(x+h)\right)  ^{2}h^{-(2\alpha+1)},\\
A  & =\left\{  (x,h):0\leq h\leq1/n,0\leq x\leq1-h\right\}  .
\end{align*}
Then
\[
\int_{0}^{1/n}\left\Vert \Delta_{h}f\right\Vert _{2}^{2}h^{-(2\alpha
+1)}dh=\int_{A}g(x,h)d(x,h)
\]
and the second integral in (\ref{secnd-integral-a}) can be written in the same
way but over a domain
\[
A^{\ast}=\left\{  (x,h):0\leq h\leq1-x,1-1/n\leq x\leq1\right\}  .
\]
Since $A^{\ast}\subset A$ and $g(x,h)\geq0$, we obtain
\begin{align*}
\left\Vert f-\bar{f}_{n}\right\Vert _{2}^{2}  & \leq2n^{-2\alpha}\int
_{A}g(x,h)d(x,h)+2n^{-2\alpha}\int_{A^{\ast}}g(x,h)d(x,h)\\
& \leq4n^{-2\alpha}\;\int_{0}^{1/n}\left\Vert \Delta_{h}f\right\Vert _{2}%
^{2}h^{-(2\alpha+1)}dh\leq4n^{-2\alpha}\;\left\vert f\right\vert
_{B_{2,2}^{\alpha}}^{2}.
\end{align*}

\end{proof}

\begin{lemma}
\label{lem-unif-approx-piececonst}For $1/2<\alpha<1$ and $f\in B_{2,2}%
^{\alpha}$ we have
\[
\left\Vert f-\bar{f}_{n}\right\Vert _{\infty}\leq C_{\alpha}\;n^{1/2-\alpha
}\;\left\Vert f\right\Vert _{B_{2,2}^{\alpha}}.
\]

\end{lemma}

\begin{proof}
For $0<\beta<1$, consider the H\"{o}lder space $C^{\beta}$ with norm
$\left\Vert f\right\Vert _{C^{\beta}}$ (cf. (\ref{holder-norm})). For $f\in
C^{\beta}$, the result
\[
\left\Vert f-\bar{f}_{n}\right\Vert _{\infty}\leq n^{-\beta}\;\left\Vert
f\right\Vert _{C^{\beta}}%
\]
is immediate. By Proposition \ref{prop-embed-theor} (ii),(i) and (iv), we have
the embeddings
\begin{equation}
B_{2,2}^{\alpha}\hookrightarrow B_{\infty,2}^{\alpha-1/2}\hookrightarrow
B_{\infty,\infty}^{\alpha-1/2}\hookrightarrow C^{a-1/2}.\label{embed-sequence}%
\end{equation}
Setting $\beta=a-1/2$, we obtain the result.
\end{proof}

\paragraph{\textbf{Periodic spaces. }}

For any $f\in L_{2}(0,1)$ and $0<\alpha<1$, let $\left\Vert f\right\Vert
_{2\,,\alpha}$ be the norm defined in terms of Fourier coefficients analogous
to (\ref{normdef}), i.e.
\[
\left\Vert f\right\Vert _{2,\alpha}^{2}:=\gamma_{f}^{2}(0)+\sum_{j=-\infty
}^{\infty}\left\vert j\right\vert ^{2\alpha}\gamma_{f}^{2}(j),\text{ where
}\gamma_{f}(j)=\int_{0}^{1}\exp(2\pi ijx)f(x)dx.
\]
Let $W^{\alpha}$ be the set of $f$ where $\left\Vert f\right\Vert _{2,\alpha
}<\infty$ equipped with this norm. This is the periodic version of the
Besov-Sobolev space $B_{2,2}^{\alpha}$ (thus a standard notation for
$W^{\alpha}$ would be $\tilde{B}_{2,2}^{\alpha}$); we will prove one part of
this claim via the embedding below. For a more comprehensive treatment cf.
Triebel (1983), Theorem 9.2.1.

\begin{lemma}
\label{lem-period-imbed}For $0<\alpha<1$ we have
\[
W^{\alpha}\hookrightarrow B_{2,2}^{\alpha}.
\]

\end{lemma}

\begin{proof}
We will first establish the inequality%
\begin{equation}
\left\Vert f\right\Vert _{B_{p,q}^{\alpha}}^{2}\leq C_{\alpha}\;\left(
\sum_{n=1}^{\infty}n^{2\alpha-1}\omega^{2}(f,n^{-1})_{2}+\left\Vert
f\right\Vert _{2}^{2}\right)  .\label{first-norm-inequal}%
\end{equation}
To this end, note that for $h\geq1$ we have $\omega^{2}(f,h)_{2}=\omega
^{2}(f,1)_{2}$ and therefore, by integrating over intervals $((n+1)^{-1}%
,n^{-1})$
\[
\int_{0}^{\infty}\left(  \frac{\omega(f,t)_{2}^{2}}{t^{2\alpha}}\right)
\frac{dt}{t}\leq\sum_{n=1}^{\infty}(n+1)^{2\alpha+1}\omega^{2}(f,n^{-1}%
)_{2}\left(  \frac{1}{n(n+1)}\right)  +\int_{1}^{\infty}\left(  \frac
{\omega(f,1)_{2}^{2}}{t^{2\alpha}}\right)  \frac{dt}{t}%
\]
which in view of $\omega(f,1)_{2}^{2}\leq4\left\Vert f\right\Vert _{2}^{2}$
gives a bound
\[
\left\vert f\right\vert _{B_{2,2}^{\alpha}}^{2}\leq2^{2\alpha}\sum
_{n=1}^{\infty}n^{2\alpha-1}\omega^{2}(f,n^{-1})_{2}+4\left\Vert f\right\Vert
_{2}^{2}\int_{1}^{\infty}t^{-(2\alpha+1)}dt.
\]
and thus (\ref{first-norm-inequal}). Define a periodic version of $\left\Vert
\Delta_{h}f\right\Vert _{2}^{2}$ by first extending the function $f$ outside
$[0,1]$ periodically, and then setting
\[
\left\Vert \tilde{\Delta}_{h}f\right\Vert _{2}^{2}:=\int_{0}^{1}\left\vert
f(x)-f(x+h)\right\vert ^{2}dx.
\]
The periodic modulus of smoothness is then
\begin{equation}
\tilde{\omega}(f,t)_{2}:=\sup_{0<h\leq t}\left\Vert \tilde{\Delta}%
_{h}f\right\Vert _{2}.\label{periodic-modulus}%
\end{equation}
Evidently we have $\omega(f,t)_{2}\leq\tilde{\omega}(f,t)_{2}$ for all $t>0 $.
Now
\begin{align*}
\left\Vert \tilde{\Delta}_{h}f\right\Vert _{2}^{2}  & =\sum_{j=-\infty
}^{\infty}\gamma_{f}^{2}(j)\left\vert 1-\exp(ijh)\right\vert ^{2}%
=4\sum_{j=-\infty}^{\infty}\gamma_{f}^{2}(j)\sin^{2}\left(  \frac{jh}%
{2}\right) \\
& \leq\sum_{|j|\leq n}\gamma_{f}^{2}(j)j^{2}h^{2}+4\sum_{j|>n}\gamma_{f}%
^{2}(j).
\end{align*}
Consequently%
\[
\sum_{n=1}^{\infty}n^{2\alpha-1}\omega^{2}(f,n^{-1})_{2}\leq\sum_{n=1}%
^{\infty}n^{2\alpha-1}\tilde{\omega}^{2}(f,n^{-1})_{2}%
\]%
\begin{align*}
& \leq\sum_{n=1}^{\infty}n^{2\alpha-1}\sum_{|j|\leq n}\gamma_{f}^{2}%
(j)j^{2}n^{-2}+4\sum_{n=1}^{\infty}n^{2\alpha-1}\sum_{|j|>n}\gamma_{f}%
^{2}(j)\\
& =\sum_{j=-\infty}^{\infty}\gamma_{f}^{2}(j)j^{2}\sum_{n\geq|j|}n^{2\alpha
-3}+4\sum_{j=-\infty}^{\infty}\gamma_{f}^{2}(j)\sum_{n<|j|}n^{2\alpha-1}\\
& \leq C_{\alpha}\;\sum_{j=-\infty}^{\infty}\gamma_{f}^{2}(j)|j|^{2\alpha
}+C_{\alpha}\;\sum_{j=-\infty}^{\infty}\gamma_{f}^{2}(j)|j|^{2\alpha}\leq
C_{\alpha}\;\left\Vert f\right\Vert _{2,\alpha}^{2}%
\end{align*}
which in conjunction with (\ref{first-norm-inequal}) proves the claim.
\end{proof}

For any $f\in L_{2}(0,1)$ (real-valued) and uneven $n$, let
\[
\tilde{f}_{n}(x)=\sum_{|j|\leq(n-1)/2}\gamma_{f}(j)\exp(2\pi ijx)
\]
be its truncated Fourier series. The letter $C$ denotes generic constants
depending on $\alpha$ but not on $f$.

\begin{lemma}
\label{lem-sobol-embed} For $1/2<\alpha<1$ we have
\begin{align*}
W^{\alpha}  & \hookrightarrow C^{\alpha-1/2},\\
\left\Vert f-\tilde{f}_{n}\right\Vert _{\infty}  & \leq Cn^{1/2-\alpha
}\;\left\Vert f\right\Vert _{2,\alpha}.
\end{align*}

\end{lemma}

\begin{proof}
The first relation follows from Lemma \ref{lem-period-imbed} and the embedding
(\ref{embed-sequence}). The second then follows from the Cauchy-Schwartz
inequality via%
\[
\left\Vert f-\tilde{f}_{n}\right\Vert _{\infty}^{2}\leq\left(  \sum
_{|j|>(n-1)/2}j^{2\alpha}\gamma_{f}^{2}(j)\right)  \left(  \sum_{|j|>(n-1)/2}%
j^{-2\alpha}\right)  \leq\left\Vert f\right\Vert _{2,\alpha}^{2}%
\;C\;n^{1-2\alpha}.
\]

\end{proof}

Let $\omega_{j,n}$ be the midpoint of the interval $W_{j,n}$ $j=1,\ldots,n$
(cf. (\ref{confer-intervals})).

\begin{lemma}
\label{lem-midpoints-appr}For $1/2<\alpha<1$ and $f\in\tilde{B}_{2,2}^{\alpha
}$ we have along uneven $n$
\[
\sum_{j=1}^{n}\left(  \tilde{f}_{n}(\omega_{j,n})-J_{j,n}(f)\right)  ^{2}\leq
C_{\alpha}\;n^{1-2\alpha}\;\left\Vert f\right\Vert _{2,\alpha}^{2}.
\]

\end{lemma}

\begin{proof}
Note that
\[
\sum_{j=1}^{n}\left(  \tilde{f}_{n}(\omega_{j,n})-J_{j,n}(f)\right)  ^{2}\leq
\]%
\[
\leq2\sum_{j=1}^{n}\left(  J_{j,n}(\tilde{f}_{n})-J_{j,n}(f)\right)
^{2}+2\sum_{j=1}^{n}\left(  \tilde{f}_{n}(\omega_{j,n})-J_{j,n}(\tilde{f}%
_{n})\right)  ^{2}.
\]
Here by Parseval's relation and the projection property of $\bar{f}_{n}$ the
first term is bounded by
\[
2n\left\Vert f-\tilde{f}_{n}\right\Vert _{2}^{2}\leq2n\left(  \frac{2}%
{n-1}\right)  ^{2\alpha}\left\Vert f\right\Vert _{2,\alpha}^{2}\leq C_{\alpha
}\;n^{1-2\alpha}\;\left\Vert f\right\Vert _{2,\alpha}^{2}.
\]
Thus it remains to show that
\begin{equation}
\sum_{j=1}^{n}\left(  \tilde{f}_{n}(\omega_{j,n})-J_{j,n}(\tilde{f}%
_{n})\right)  ^{2}\leq C_{\alpha}\;n^{1-2\alpha}\;\left\Vert f\right\Vert
_{2,\alpha}^{2}.\label{remains-1}%
\end{equation}
We have
\begin{align*}
\left\vert \tilde{f}_{n}(\omega_{j,n})-J_{j,n}(\tilde{f}_{n})\right\vert  &
\leq n\int_{W_{j,n}}\left\vert \tilde{f}_{n}(x)-\tilde{f}_{n}(\omega
_{j,n})\right\vert dx\leq n\int_{W_{j,n}}\left\vert \int_{x}^{\omega_{j,n}%
}D\tilde{f}_{n}(t)dt\right\vert dx\\
& \leq n\int_{W_{j,n}}\left\vert x-\omega_{j,n}\right\vert ^{1/2}\left(
\int_{x}^{\omega_{j,n}}\left(  D\tilde{f}_{n}(t\right)  ^{2}dt\right)
^{1/2}dx\\
& \leq n^{-1/2}\left(  \int_{W_{j,n}}\left(  D\tilde{f}_{n}(t\right)
^{2}dt\right)  ^{1/2}.
\end{align*}
Consequently%
\[
\sum_{j=1}^{n}\left(  \tilde{f}_{n}(\omega_{j,n})-J_{j,n}(\tilde{f}%
_{n})\right)  ^{2}\leq n^{-1}\int_{[0,1]}\left(  D\tilde{f}_{n}(t\right)
^{2}dt.
\]
By termwise differentiation and Parseval's relation the right side equals
\begin{align*}
n^{-1}\sum_{\left\vert j\right\vert \leq(n-1)/2}j^{2}\gamma_{f}^{2}(j)  &
=\sum_{\left\vert j\right\vert \leq(n-1)/2}\left(  \frac{|j|^{2-2\alpha}}%
{n}\right)  |j|^{2\alpha}\gamma_{f}^{2}(j)\\
& \leq n^{1-2\alpha}\left\Vert f\right\Vert _{2,\alpha}^{2}%
\end{align*}
which establishes (\ref{remains-1}).
\end{proof}

\begin{remark}
\label{rem-periodic-besov}Periodic Besov spaces $\tilde{B}_{p,q}^{\alpha}$.
\emph{These can be defined for }$0<\alpha<1$\emph{\ and }$1\leq p,q\leq\infty
$\emph{\ analogously to the spaces }$B_{p,q}^{\alpha}$ \emph{as above, using
an periodic increment norm }$\left\Vert \tilde{\Delta}_{h}f\right\Vert _{p}%
$\emph{\ and a periodic modulus of smoothness }$\tilde{\omega}(f,t)_{p}%
$\emph{\ defined analogously to (\ref{periodic-modulus}). Clearly then
}$\tilde{B}_{p,q}^{\alpha}\hookrightarrow B_{p,q}^{\alpha}$\emph{. An
intrinsic characterization in terms of Fourier coefficients }$\gamma_{f}%
(k)$\emph{\ is as follows: for }$p\geq2$\emph{\ and }$1/2<\alpha<1$\emph{\ the
expression }%
\[
\left(  \sum_{j=0}^{\infty}2^{j\alpha q}\left\Vert \sum_{2^{j-1}-1<\left\vert
k\right\vert <2^{j}}\gamma_{f}(k)\exp(2\pi ikx)\right\Vert _{p}^{q}\right)
^{1/q}%
\]
\emph{is an equivalent norm in }$\tilde{B}_{p,q}^{\alpha}$\emph{; cf.
Nikolskii, sec. 5.6, relation (6) (cp. also Triebel (1983), definition
2.3.1/2). }.
\end{remark}

\begin{tabular}
[t]{ll}%
\textsc{Universit\'{e} de Provence} & \textsc{Department of Mathematics}\\
\textsc{CMI} & \textsc{Malott Hall}\\
\textsc{39, rue F.Joliot-Curie} & \textsc{Cornell University}\\
\textsc{13453 Marseille, Cedex 13, France} & \textsc{Ithaca NY 14853}\\
\textsc{e-mail:}
golubev@gyptis.univ-mrs.fr\ \ \ \ \ \ \ \ \ \ \ \ \ \ \ \ \ \  &
\textsc{e-mail:} nussbaum@math.cornell.edu
\end{tabular}

\bigskip%

\begin{tabular}
[t]{ll}%
\textsc{Department of Statistics} & \\
\textsc{Yale University} & \\
\textsc{P.O. Box 208290} & \\
\textsc{New Haven CT 06520} & \\
\textsc{e-mail:} huibin.zhou@yale.edu\ \ \ \ \ \ \ \ \ \ \ \ \ \ \ \ \ \  &
\end{tabular}

\end{document}